\documentclass[preprint,12pt]{elsarticle}

\usepackage{graphicx}
\usepackage{amsmath}
\usepackage{amsfonts}
 \usepackage{amsthm}
\usepackage{enumerate}
\usepackage{url}
\usepackage{color}
\usepackage{natbib}
\usepackage{epsfig}
\usepackage{ifthen}
\newcommand{\comment}[1]{}

\newcommand{\ind}{{\bf 1}}

\def\inddd#1{1_{\left\{#1\right\}}}

\newcommand{\proba}{\mathbb P}

\newcommand{\esp}{{\mathbb E}}

\newcommand{\inv}{^{-1}}

\newcommand{\eqnh}{\begin{eqnarray*}}
\newcommand{\eqne}{\end{eqnarray*}}
\newcommand{\eqnhn}{\begin{eqnarray}}
\newcommand{\eqnen}{\end{eqnarray}}
\newcommand{\equh}{\begin{equation}}
\newcommand{\eque}{\end{equation}}

\def\summ#1#2#3{\sum_{#1 = #2}^{#3}}
\def\prodd#1#2#3{\prod_{#1 = #2}^{#3}}
\def\sif#1#2{\sum_{#1=#2}^\infty}

\newcommand{\eqd}{\stackrel{d}{=}}

\def\topp#1{^{(#1)}}

\def\ccbb#1{\left\{#1\right\}}
\def\sccbb#1{\{#1\}}
\def\pp#1{\left(#1\right)}

\def\bb#1{\left[#1\right]}
\def\sbb#1{[#1]}
\def\mmid{\;\middle\vert\;}

\def\floor#1{\left\lfloor #1 \right\rfloor}

\def\vv#1{{\boldsymbol #1}}

\def\qmand{\quad\mbox{ and }\quad}

\def\mwith{\mbox{ with }}
\def\qmwith{\quad\mbox{ with }\quad}
\def\mfa{\mbox{ for all }}

\def\mmas{\mbox{ as }}

\def\wt#1{\widetilde{#1}}

\def\what#1{\widehat{#1}}

\def\limn{\lim_{n\to\infty}}
\def\limm{\lim_{m\to\infty}}

\def\limsupn{\limsup_{n\to\infty}}

\def\weakto{\Rightarrow}

\def\R{{\mathbb R}}

\def\N{{\mathbb N}}

\def\Q{{\mathbb Q}}
\def\T{{\mathbb T}}

\def\calB{\mathcal B}

\def\calD{\mathcal D}
\def\calE{\mathcal E}

\def\filF{\mathcal F}
\def\calG{\mathcal G}
\def\calI{\mathcal I}
\def\calK{\mathcal K}
\def\calH{\mathcal H}
\def\calM{\mathcal M}
\def\calN{\mathcal N}

\def\vF{{\mathbf F}}

\def\Zab{Z_{\alpha,\beta,p}}

\newtheorem{Thm}{Theorem}[section]

\newtheorem{Lem}[Thm]{Lemma}
\newtheorem{Cor}[Thm]{Corollary}
\newtheorem{Pro}[Thm]{Proposition}
 
\theoremstyle{definition}
\newtheorem{Def}{Definition}[section]

\theoremstyle{remark}
\newtheorem{Rem}{Remark}[section]

\theoremstyle{remark}
\newtheorem{Eg}{Example}[section]

\newtheorem{assump}{Assumption}

\newcommand{\EqD}{\overset{d}{=}}
\newcommand{\ConvFDD}{\overset{f.d.d.}{\longrightarrow}}
\newcommand{\cl}{\mathcal}
\newcommand{\lf}{\lfloor}
\newcommand{\rf}{\rfloor}

\newcommand{\wh}{\widehat}
\def\supp{\mathrm{supp}}
\newcommand{\RV}{\mathrm{RV}}

\def\beqn{\begin{equation}}
\def\beqn*{$$}
\def\eeqn{\end{equation}}

\def\E{\mathbb{E}}

\newcommand{\That}{\widehat T}

\newcommand{\hqbeta}{h_q^{(\beta)}}
\newcommand{\hqbetaeta}{h_q^{(\beta-\eta)}}
\def\I{\mathcal I}
\newcommand{\Ii}{\mathcal I(i)}
\newcommand{\abIi}{|\mathcal I(i)|}
\newcommand{\hIibeta}{h_{|\I (i)|}^{(\beta)}}
\newcommand{\xIi}{\vv x_{\I(i)}}
\newcommand{\Kil}{\mathcal K(i,\ell)}
\newcommand{\Gammabeta}{\Gamma (\beta) \Gamma (2-\beta)}

\usepackage{amssymb}

\journal{Stochastic Processes and their Applications}

\begin{document}

\begin{frontmatter}



\title{A functional non-central limit theorem for multiple-stable 
processes with long-range dependence}


\author[label1]{Shuyang Bai}

\author[label2]{Takashi Owada}

\author[label3]{Yizao Wang}

\address[label1]{Department of Statistics, University of Georgia, 310 Herty Drive, 
Athens, GA, 30602, USA. {bsy9142@uga.edu}}

\address[label2]{Department of Statistics, Purdue University, 250 N.~University Street, 
West Lafayette, IN, 47907, USA. {owada@purdue.edu}}

\address[label3]{Department of Mathematical Sciences, University of Cincinnati, 2815 Commons Way, Cincinnati, OH, 45221-0025, USA. {yizao.wang@uc.edu}}

\begin{abstract}
A functional limit theorem is established for the partial-sum process of a class of stationary sequences which exhibit both heavy tails and long-range dependence. The stationary sequence is constructed using multiple stochastic integrals with heavy-tailed marginal distribution. Furthermore, the multiple stochastic integrals are built upon a large family of dynamical systems that are ergodic and conservative, leading to the long-range dependence phenomenon of the model. The limits constitute a new class of self-similar processes with stationary increments. They are represented by multiple stable integrals, where the integrands involve the local times of intersections of independent stationary stable regenerative sets. 
\end{abstract}

%


\end{frontmatter}


\section{Introduction}
\subsection{Background}
The seminal work of Rosi\'nski \citep{rosinski95structure}
revealed an intriguing connection between stationary stable processes and ergodic theory. 
Consider a stationary process in the form of
\equh\label{eq:1}
X_k = \int_E f(T^k x)M(dx), \ \ k\in\N,
\eque
where $M$ is symmetric $\alpha$-stable random measure 
 on a  measure space $(E,\calE,\mu)$, $f:E\to \R$ is a measurable function and $T$ is a measure-preserving transform from $E$ to $E$. Then, many   properties of the process $X$ can be derived from the underlying dynamical system $(E,\calE,\mu,T)$. 
 Because of this connection, the process $X$ is also referred to as {\em driven by the flow $T$}, and many developments on structures, representations, and ergodic properties of such processes have stemmed from this connection (see e.g., \citep{samorodnitsky16stochastic,samorodnitsky05null,kabluchko16stochastic,pipiras02structure,pipiras02decomposition,pipiras17stable,sarkar18stable,roy08stationary,roy10nonsingular,wang13ergodic};
  background to be reviewed in Section \ref{sec:ergodic}).  
In particular, it was argued by Samorodnitsky \citep[Remark 2.5]{samorodnitsky05null} that the case where $T$ is conservative and ergodic is the most challenging to develop a satisfactory characterization of the ergodic properties 
of the process
 in terms of the 
 underlying 
 dynamical system.

While examples of stable processes driven by conservative and ergodic flows  have been known for more than 20 years since \citep{rosinski96classes}, limit theorems for such processes have not been established until in very recent breakthroughs in a series of papers by Samorodnitsky and coauthors \citep{owada15functional,owada15maxima,lacaux16time,samorodnitsky17extremal}, all exhibiting phenomena of long-range dependence with new limit objects. Here, by long-range dependence, we mean generally that the partial-sum process $(S_{\floor {nt}})_{t\in[0,1]}$, with $S_n:=X_1+\cdots+X_n$, scales to a non-degenerate stochastic process with a normalization that is different from the case when $(X_k)_{k\in\N}$ are i.i.d. We follow this point of view as in Samorodnitsky \citep{samorodnitsky16stochastic}, and one could also consider limit theorems for other statistics; the key is always the abnormal normalization compared to the i.i.d.~case. 

The functional central limit theorem for stationary stable processes driven by a conservative and ergodic flow, established in  \citep{owada15functional}, serves as our starting point and takes the following form.
With $f$ in \eqref{eq:1} such that the support has finite $\mu$-measure and $\mu(f):=\int_Efd\mu$ is finite and nonzero, it was shown that
\equh\label{eq:OS}
\frac1{d_n}\pp{S_{\floor{nt}}}_{t\in[0,1]}
\weakto \mu(f)\pp{\int_{\Omega'\times [0,\infty)}\calM_{\beta}((t-v)_+,\omega')S_{\alpha,\beta}(d\omega',dv)}_{t\in[0,1]}
\eque
in $D([0,1])$, 
 where $\alpha\in(0,2)$, $\beta \in (0,1)$, and $d_n$ is a regularly varying sequence with exponent ${\beta+(1-\beta)/\alpha}$. (This was actually established in a slightly more general framework with $M$ replaced by an infinitely-divisible random measure with heavy-tail index $\alpha$.)
Here, 
$(\Omega',\filF',P')$ is a probability space separate from the one that carries the randomness of the stochastic integral itself, 
$S_{\alpha,\beta}$ is a symmetric $\alpha$-stable (S$\alpha$S) random measure on $\Omega' \times [0,\infty)$ with control measure $P' \times (1-\beta)v^{-\beta}dv$, 
 and $\calM_\beta$ is the Mittag--Leffler process with index $\beta$, the inverse process of a $\beta$-stable subordinator, defined on $(\Omega',\filF',P')$.

Here, $\beta\in(0,1)$ is the memory parameter of an underlying dynamical system (see Section \ref{sec:nCLT} and in particular how $\beta$ characterizes the memory of $T$ in terms of Assumption \ref{assump}), and as $\beta\downarrow0$ the limit process in \eqref{eq:OS} becomes an S$\alpha$S L\'evy process. At the core of this result, the appearance of the Mittag--Leffler process is established as a functional generalization of the one-dimensional Darling--Kac limit theorem  in 
\citep{aaronson81asymptotic,bingham71limit} for the underlying dynamical system, which is of independent interest in ergodic theory. Later developments \citep{lacaux16time,samorodnitsky17extremal} revealed that  more essentially, stable regenerative sets \citep{bertoin99subordinators} and their intersections play a fundamental role in describing the limit objects for a large family of processes driven by conservative and ergodic flows.

In this paper, as a generalization of \eqref{eq:1} we consider the process defined in terms of multiple stochastic integrals in the form of 
\equh\label{eq:2}
X_k = \int_{E^p}'f(T^kx_1,\dots,T^kx_p)M(dx_1)
\cdots
 M(dx_p), \ \ k\in\N, \ \ p\in\N,
\eque
where the prime mark $'$ indicates that the multiple integral is defined to {\em exclude the diagonals}, and this time $f$ is a measurable function from $E^p$ to $\R$. The definition of multiple stochastic integrals will be recalled  in Section \ref{sec:mult int} below. 

  We restrict to the case of multiple integrals without the diagonals, in order to obtain limit  processes in the form of  {\em multiple stable integrals}, which we refer to as {\em multiple-stable processes}. 
Since the seminal works of Dobrushin and Major \citep{dobrushin79noncentral} and Taqqu \citep{taqqu79convergence},   the processes in the form of multiple Gaussian integrals 
have frequently appeared in limit theorems   under long-range dependence. For example, they were obtained as limits for partial sums   \citep{dobrushin79noncentral,taqqu79convergence, surgailis82domains,arcones1994limit,ho97limit,bai14generalized}), for empirical processes \citep{dehling1989empirical,ho1996asymptotic,wu2003empirical}  as well as for quadratic forms \citep{fox1985noncentral,terrin1990noncentral}.
  Such limit theorems  are often referred to as {\em non-central limit theorems} and have found numerous applications to statistical theories for long-range dependent data (see, e.g., \cite{beran13long} and the references therein).    Limit theorems with (non-Gaussian) multiple-stable processes as limits, to the best of our knowledge however, have been rarely considered so far in the literature of long-range dependence.
Note that the exclusion of the diagonals is necessary to obtain 
multiple-stable processes with multiplicity $p\ge 2$:
 with the terms on the diagonal included, the case $p=2$ has been partly considered in \citep{owada16limit}, and the limit is again a stable process.

\subsection{Overview of main results}

Our ultimate goal (Theorem \ref{Thm:CLT}) is to establish formally that 
\[
\frac1{d_n}\pp{\summ k1{\floor{nt}}X_k}_{t\in[0,1]}\weakto \pp{\Zab(t)}_{t\in[0,1]}
\]
for a large family of $(X_k)$ in \eqref{eq:2}, and the limit process has the representation
\begin{multline}\label{eq:Z}
\pp{\Zab(t)}_{t\ge 0}\\
 \eqd \pp{\int_{(\vF\times
[0,\infty)
)^p}'L_t\pp{\bigcap_{i=1}^p(R_i+v_i)}S_{\alpha,\beta}(d R_1,dv_1)\cdots S_{\alpha,\beta}(dR_p,dv_p)}_{t\ge 0},
\end{multline}
where $S_{\alpha,\beta}$ is an S$\alpha$S random measure on $\vF\times [0,\infty)$, with control measure $P_\beta\times (1-\beta)v^{-\beta}dv$, with $P_\beta$ the probability measure on $\vF\equiv \vF([0,\infty))$, the space of closed subsets of $[0,\infty)$,  induced by the law of a  $\beta$-stable regenerative set, and $L_t$ is the local-time functional for a $(p\beta-p+1)$-stable regenerative set 
\citep{kingman73intrinsic}.  

An immediate observation is that for the right-hand side of \eqref{eq:Z} to be non-degenerate, we need $\bigcap_{i=1}^p(R_i+v_i)$ to be non-empty, with $(R_i)_{i=1,\dots,p}$ being i.i.d.~$\beta$-stable regenerative sets. The key relation between the memory parameter $\beta$ and the multiplicity $p$ assumed  throughout this paper is that
\equh\label{eq:beta}
\beta\in(0,1), \quad p\in\N \quad \mbox{ such that }\quad \beta_p:=p\beta-p+1 \in(0,1),
\eque
or equivalently $\beta\in(1-1/p,1)$. It is known (e.g., \citep{samorodnitsky17extremal})  that this is exactly the case when $\bigcap_{i=1}^p(R_i+v_i)$ is a $\beta_p$-stable regenerative set with a random shift with probability one. 
 When \eqref{eq:beta} is violated and $v_i$ are all different,
  the intersection becomes an empty set with probability one  and hence $\Zab$ becomes degenerate. The  limit  theorem in such a case will be of a different nature and  addressed in a separate paper.

Our theorem applies to a large family of dynamical systems, including in particular the  shift transforms of certain null-recurrent Markov chains, and a class of transforms on the real line called  the AFN-systems \citep{zweimuller98ergodic,zweimuller00ergodic}
often considered in the literature of infinite ergodic theory.
Establishing the aforementioned convergence, however,  turns out to be a completely different task from the one in
\citep{owada15functional},
and the proof consists of two parts.
The first part  is devoted to the investigation of the integrand of the right-hand side of \eqref{eq:Z}, which are local-time processes of intersections of stable regenerative sets (Section \ref{sec:local times}).   Let $(R_i)_{i\in\N}$ be i.i.d.~$\beta$-stable regenerative sets. To exploit a series representation of the multiple integral \eqref{eq:Z} (see \eqref{eq:Z_t [0,1]} below),   we need to characterize the law of 
\[
L_{I,t} \equiv L_t\pp{\bigcap_{i\in I}(R_i+v_i)} \mfa I\subset \N,~ |I| = p,~t\ge 0,
\]
jointly in $I$ and $t$, governed by certain law on the shifts $(v_i)_{i\in I}$ independent from the regenerative sets.
Marginally, for each $I$, $(L_{I,t})_{t\ge 0}$ has the law of a   Mittag--Leffler process shifted in time with parameter $\beta_{p}$, up to a multiplicative constant \citep{samorodnitsky17extremal}. 
In particular when $p=1$ we have
\equh\label{eq:p=1}
\pp{L_t(R_1+v_1)}_{t\ge 0} \eqd c_\beta\pp{\calM_{\beta}((t-v_1)_+)}_{t\ge 0}
\eque
for some constant $c_\beta$. It is then a matter of convenience to work with either of the two representations in \eqref{eq:p=1}, and the right-hand side was used in    \citep{owada15functional}.  
However when $p\ge 2$, the information from the Mittag--Leffler process is only marginal, whereas  we need to work with $L_{I,t}$ jointly in $I,t$. More precisely, 
we shall compute all their joint moments with appropriately randomized shifts. For this key calculation, we adapt the {\em random covering scheme} for  constructing regenerative sets \citep{fitzsimmons85intersections}, to develop approximations of joint law of $L_{I,t}$   in Theorem \ref{thm:1}.

The second part of the proof is devoted to the convergence of the partial-sum process to $\Zab$.  To illustrate the idea, assume for simplicity that   $f(x_1,\ldots,x_p)=1_A(x_1)\ldots 1_A(x_p)$, where $A$ is a suitable finite-measure subset of $E$.  To work with a series representation of  the multiple integral \eqref{eq:2} (see \eqref{eq:X_k series} below), the key ingredient is to show  the joint convergence after proper normalization,  in $I$ and $t$, of counting processes of simultaneous returns of i.i.d.~dynamical systems, indexed by $i\in I$, in the form 
\equh\label{eq:key}
\sum_{k=1}^{\floor{nt}}\prod_{i\in I}\inddd{T^kx_i\in A},
\eque
where the staring points $x_i\in E$ are governed by i.i.d.\  infinite stationary distributions. For any individual $I$, 
our assumptions essentially entail that the simultaneous-return times behave like renewal times of a heavy-tailed renewal process, and then the above is known to converge to the local-time process $L_{I,t}(R^*+V^*)$ for $\beta_p$-stable regenerative set $R^*$ with a random shift $V^*$. This certainly includes $p=1$ as a special case (\citep{bingham71limit} and 
\citep[Theorem 6.1]{owada15functional}). The challenge lies in characterizing the joint limits for say $(I_j,t_j)_{j=1,\dots,r}$. Theorem \ref{Thm:local time} is devoted to this task, showing that the limit of the above is  
$(L_{I_j,t_j})_{j=1,\dots,r}$
 (with respect to random shifts 
 $v_j$). The proof is of combinatorial nature and by computing the asymptotic moments of \eqref{eq:key}. A delicate approximation scheme similar to Krickeberg \citep{krickeberg67strong} is then developed so that the asymptotic moment formula is  extended  to the case where the product in \eqref{eq:key} is replaced by $f(T^kx_1,\dots,T^k x_{p})$ for a general class of functions of $f$.

 We also mention that  a simultaneous work \citep{bai20limit} considers the case  where  the random measure $M$ in \eqref{eq:2} is replaced by a Gaussian one  so that $X_k$ has  finite variance marginally. In that case,   a functional non-central limit theorem is established with  Hermite processes (e.g., \cite{taqqu79convergence}), a well-known class of processes    represented by multiple Gaussian integrals,  arising as limits. It is remarkable that the proof techniques of \citep{bai20limit} exploit special properties of multiple Gaussian integrals,
and in particular, the local-time processes and their approximations as we deal with here are not needed in \citep{bai20limit}.
On the other hand, however, the joint local-time processes are still intrinsically connected to the limit Hermite processes.
As shown in the manuscript \cite{bai19representations} after the present work, if the multiple-stable integrals in \eqref{eq:Z} are extended
to the Gaussian case $\alpha=2$, then they yield   new representations for the Hermite processes.

 The paper is organized as follows. Section \ref{sec:local times} introduces the joint local-time processes, and 
 establishes
  a formula for the joint moments by the random covering 
 scheme. Section \ref{sec:limit process} reviews certain series representations of multiple integrals and defines formally the limit process $\Zab$. Section \ref{sec:nCLT} introduces our model of stationary processes in terms of multiple integrals with long-range dependence, and states the main non-central limit theorem. Section \ref{sec:proof} is devoted to the proof of the main theorem. Throughout the paper, $C$ and $C_i$  denote     generic positive constants which are independent of $n$ and may change from line to line.

\section{Local-time processes}\label{sec:local times}
\subsection{Definitions and results}
We start by recalling some facts about random closed sets on $[0,\infty)$, and in particular, stable regenerative sets.  We refer the reader to \citep{molchanov17theory} for more details. Let $\vF \equiv \mathbf{F}([0,\infty))$ denote the collection of all closed subsets of $[0,\infty)$. We equip $\vF$ with the Fell topology  which is  generated by the sets $\{F\in \vF: F\cap G \neq \emptyset\}$ and $\{F\in \vF: \ F\cap K=\emptyset\}$ for arbitrary open $G\subset [0,\infty)$ and compact $K\subset [0,\infty)$.    A  random closed set on $[0,\infty)$ is a Borel measurable  random element taking values in $\vF$.  If the law of a random closed set $R$ on $[0,\infty)$ is identical to that of the closed range of  a subordinator \citep{bertoin99subordinators}, then $R$ is said to be a regenerative set. The random set $R$ is, in addition, said to be $\beta$-stable, $\beta\in(0,1)$, if the corresponding subordinator, say $(\sigma_t)_{t\ge0}$, is  
  $\beta$-stable; that is, $(\sigma_t)_{t\ge0}$ is a non-decreasing L\'evy process determined by 
\begin{equation}\label{eq:laplace}
\E e^{-\lambda \sigma_t}=\exp(-  t\lambda^\beta), \lambda\ge0.
\end{equation}
In this case, the associated L\'evy measure of the regenerative set $R$ is
\begin{equation}\label{eq:Pi_beta}
\Pi_\beta(dx)=\frac{ \beta}{\Gamma(1-\beta)}x^{-1-\beta} 1_{(0,\infty)}(x) dx,
\end{equation}
which characterizes the law of $R$.

For our purposes, we shall work with a family of countably many independent stable regenerative sets with independent shifts, and we need in particular to describe their intersections. Let $(R_i)_{i\in\N}$ be i.i.d.~$\beta$-stable regenerative sets and $(V_i)_{i\in\N}$ be independent random shifts with arbitrary laws, and the two sequences are independent. 
Under our assumption on $\beta$ and $p$ in \eqref{eq:beta},  for every 
\begin{equation}\label{eq:D_p}
I\in \calD_p:=\ccbb{I=(i_1,\ldots,i_p)\in\N^p:~i_1<
\cdots<i_p},
\end{equation}
we have
\begin{equation}\label{eq:decomp intersect}
\bigcap_{i\in I}(R_i+V_i) \eqd R_I+V_I  ,
\end{equation}
where $R_I$ is a $\beta_p$-stable regenerative set and $V_I$ is an independent random variable. In words, the intersection of $p$ independent randomly shifted $\beta$-stable regenerative set is $\beta_p$-stable regenerative with an independent random shift. This follows for example from the strong Markov property of the regenerative sets. See also \citep[Appendix B]{samorodnitsky17extremal}.

There are multiple ways to construct the local time associated to a regenerative set (\citep[Chapter 12]{kallenberg17random}). 
For the series representation of multiple integrals needed later, 
we use a construction due to Kingman \citep{kingman73intrinsic} which treats the local time as a functional defined on $\mathbf{F}$. In particular, set
\[
L=L^{(\beta_p)}: \mathbf{F}\rightarrow[0,\infty], ~ L(F) :=\limsupn\frac1{l_{\beta_p}(n)}\lambda\left(F+\bb{- \frac{1}{2n},\frac{1}{2n}}\right),
\]
where $\lambda$ is the Lebesgue measure, 
$F+\sbb{- 1/{2n},1/{2n}}\equiv\cup_{x\in F}[x-1/{2n},x+1/{2n}]$,
 and the   normalization sequence
\[
l_{\beta_p}(n)=\int_{0}^{1/n} \Pi_{\beta_p}((x,\infty))dx= \frac{   n^{\beta_p-1}}{\Gamma(2-\beta_p)},
\]
where $\Pi_\beta$ is as in \eqref{eq:Pi_beta}. 
The exclusive choice of $\beta_p$ as in \eqref{eq:beta} is due to the fact that we shall only deal with local times of shifted $\beta_p$-stable regenerative sets, obtained as the intersection of $p$ independent stable regenerative sets. 
We then define 
\begin{equation}\label{eq:L_t}
L_t(F):=L(F\cap [0,t]), \quad t\ge 0.
\end{equation}
\begin{Lem}
The functionals $L$ and $L_t$ are $\calB(\vF)/\cl{B}( [0,\infty])$-measurable, where $\calB(\vF)$ and $\cl{B}( [0,\infty])$ denote the Borel $\sigma$-fields on $\vF$ and $  [0,\infty]$ respectively.
\end{Lem}
\begin{proof}
 Direct sum  and intersection   are measurable operations for closed sets   \citep[Theorem 1.3.25]{molchanov17theory}.   
 The Lebesgue measure $\lambda$ is also a measurable functional from  $\mathbf{F}$ to 
 $[0,\infty]$.
  Indeed,  write $[0,\infty)=\cup_{n=0}^\infty K_n$ where $K_n=[n,n+1]$.  Then $F\mapsto \lambda (F\cap K_n)$ is a measurable mapping from $\mathbf{F}$ to $[0,\infty]$ since it is upper semi-continuous \citep[Proposition E.13]{molchanov17theory}. Hence   $F\rightarrow \lambda({F})=
  \sif n0
   \lambda(F\cap K_n)$ is measurable as well.
\end{proof}
From now on,  we denote the local-time processes using the notation
\equh\label{eq:LIt}
L_{I,t} \equiv L_t\pp{\bigcap_{i\in I}(R_i+V_i)},~ t \in[0,\infty), ~I\in \calD_p. 
\eque
In view of \eqref{eq:decomp intersect} and \cite[Theorem 3]{kingman73intrinsic} (conditioning on $V_I$ in \eqref{eq:decomp intersect}), for each $I\in\calD_p$, the finite-dimensional distributions of $(L_{I,t})_{t\ge 0}$   coincide  with those of a randomly shifted $\beta_p$-Mittag--Leffler process, $(\calM_{\beta_p}(t-V_I)_+)_{t\ge 0}$, where   $V_I$ is independent of $\calM_{\beta_p}$. In particular,  $(L_{I,t})_{t\ge 0}$ admits a version which has a non-decreasing and continuous path a.s..

The advantage of the above construction is that now for different $I,t$, the corresponding local times are constructed on a common probability space as  measurable functions  evaluated at intersections of independent shifted random regenerative sets.
We shall develop the formula for their joint moments. We   work with a specific choice of the random shifts: most of the time we assume in addition that  $(V_i)_{i\in\N}$ are i.i.d.~with the law 
\equh\label{eq:V}
P(V_i\le v) = v^{1-\beta}, v\in[0,1].
\eque

\begin{Rem}\label{rem:stationary} The law of the shift \eqref{eq:V} will show up naturally in our limit theorem later.  To understand the origin of  \eqref{eq:V}, recall that a random closed set $F$ on $[0,\infty)$ is said to be stationary, if   its law is unchanged under the map  $F\rightarrow (F\cap [x,\infty))-x$ for any $x>0$. While a $\beta$-stable regenerative set $R_i$ itself is not stationary, it is known that with an independent shift $V_i$ following an infinite law  proportional to $v^{-\beta}dv$ on $\R_+$, the shifted   random  (with respect to an infinite measure) set $R_i+V_i$ is stationary (\citep[Proposition 4.1]{lacaux16time}, see also \citep{fitzsimmons88stationary}).  The law  \eqref{eq:V} is  nothing but the normalized restriction to $[0,1]$ of this infinite law.  As a consequence, one could derive that $\bigcap_{i\in I}(R_i+V_i) \equiv R_I+V_I$ is also stationary   with respect to an infinite 
  measure  
 \citep[Corollary B.3]{samorodnitsky17extremal}. This is in accordance with the stationarity of the increments of  the process $\Zab$ in \eqref{eq:Z} (see Section \ref{sec:limit process def}).  
\end{Rem}

From now on we fix 
$\beta\in(0,1)$, $p\in\N$,
such that \eqref{eq:beta} holds. 
  Introduce for $q\ge 2$, a symmetric function $h_q\topp\beta$ on the off-diagonal subset of $(0,1)^q$   determined by
\equh\label{eq:hq}
h_q\topp\beta(x_1,\dots, x_q) = \Gamma(\beta)\Gamma(2-\beta)\prodd j2q (x_j-x_{j-1})^{\beta-1},\  0<x_1<\cdots<x_q<1.
\eque
Here and below, 
for any $q\in\N$, a $q$-variate function $f$ is said to be symmetric, if $f(x_1,\dots,x_q) = f(x_{\sigma(1)},\dots,x_{\sigma(q)})$ for any permutation $\sigma$ of $\{1,\dots,q\}$.
For a symmetric function on the off-diagonal set, we do not specify the values on the diagonal set 
$\{(x_1,\dots,x_q)\in(0,1)^q: x_i = x_j \mbox{ for some $i\ne j$}\}$, which has zero Lebesgue measure and hence does not have any impact 
in our derivation.
Introduce also $h_0\topp\beta:=1$ and $h_1\topp\beta(x):=\Gamma(\beta)\Gamma(2-\beta)$. 
The main result of this section is the following.
\begin{Thm}\label{thm:1}
Let $(R_i)_{i\in\N}$ be i.i.d.~$\beta$-stable regenerative sets and $(V_i)_{i\in\N}$ be i.i.d.~with law \eqref{eq:V}, the two sequences being independent. Given a collection of $I_\ell \in \cl{D}_p $, $\ell=1,\ldots,r$, set 
$
K  = \max\pp{\bigcup_{\ell=1}^r I_\ell}.
$ Then, for all   $\vv{t}=(t_1,\ldots,t_r)\in [0,1]^r$,  
\equh\label{eq:moment limit}
\E \pp{\prod_{\ell =1}^r L_{I_\ell,t_\ell}}=  \frac{1}{\Gamma(\beta_p)^r}  \int_{\vv 0 < \vv x <\vv  t} \prod_{i=1}^K \hIibeta (\xIi) \, d\vv x
\eque
with 
\begin{equation}\label{eq:calI}
\calI(i) := \ccbb{\ell\in\{1,\dots,r\}:i\in I_\ell},~ 
i=1,\dots,K.
\end{equation}
\end{Thm}
Above and below, we write $\vv x = (x_1,\dots,x_r), d\vv x = dx_1\dots dx_r$,
$\vv 0 = (0,\dots,0), \vv 1 = (1,\dots,1)$, and $\vv x<\vv y$ is understood in the coordinate-wise sense. Also, write
\[
\xIi = (x_\ell)_{\ell\in \calI(i)},
\]
understood as the vector in $\R_+^{|\calI(i)|}$. (Since each $h_{|\calI(i)|}\topp\beta$ is a symmetric function, the order of coordinates of $\xIi$ is irrelevant here.)

Write $\vv V_I = (V_i)_{i\in I}$ and $\vv R_I = (R_i)_{i\in I}$. In view of \eqref{eq:LIt}, 
from now on we write explicitly 
$L_{I,t} \equiv L_{I,t}(\vv R_I,\vv V_I)$. We have, by Fubini's theorem,
\[
\E \pp{\prod_{\ell =1}^r L_{I_\ell,t_\ell}(\vv R_{I_\ell},\vv V_{I_\ell})}= \int_{(0,1)^K}\esp\pp{\prodd\ell1rL_{I_\ell,t_\ell}(\vv R_{I_\ell},\vv v_{I_\ell})}(1-\beta)^K\prodd i1Kv_i^{-\beta}d\vv v.
\]
We shall  establish a formula for
\[
\Psi(\vv v):= \esp\pp{\prodd\ell1rL_{I_\ell,t_\ell}(\vv R_{I_\ell},\vv v_{I_\ell})}, \mfa \vv v\in (0,1)^K,
\]
where the expectation is with respect to   the randomness coming from $\vv R_{I_\ell}$, $\ell=1,\ldots,r$. 
At the core of our argument is the following proposition.
Let $g_q$, $q\in \N$ be  symmetric functions on the off-diagonal subset of $(0,1)^q$ such that
\equh\label{eq:g_q}
g_q^{(\beta)}(x_1,\ldots,x_q)= \prod_{j=1}^q (x_j - x_{j-1})^{\beta-1}, \ \ x_0:=0 < x_1 < \cdots < x_q <1, 
\eque
and $g_0^{(\beta)} := 1$.
 We write $\max(\vv v_I) = \max_{i\in I}v_i$, and similarly for $\min(\vv v_I)$.

\begin{Pro}\label{prop:psy} Under the assumption of Theorem \ref{thm:1},
\begin{equation}\label{eq:psy}
\Psi(\vv v)
=  \frac{1}{\Gamma(\beta_p)^r}   \int_{\max(\vv v_{I_\ell}) < x_\ell < t_\ell, \, \ell=1,\dots,r} \prod_{i=1}^K  g^{(\beta)}_{|\cl{I}(i)|}(\xIi-v_{i}\vv 1)d\vv x.
\end{equation}
In particular, $\Psi(\vv v) = 0$ if $\max(\vv v_{I_\ell})\ge 
t_\ell$ for some $\ell=1,\dots,r$. 
\end{Pro}
The proof of the proposition 
is postponed to
 Section \ref{sec:covering}  below. 

\begin{proof}[Proof of Theorem \ref{thm:1}] 
We shall compute
\equh\label{eq:convolution0}
(1-\beta)^{K}\int_{(0,1)^K}\Psi(\vv v) \prod_{i=1}^K v_i^{-\beta}d \vv v.
\eque
We express the constraint $\max(\vv v_{I_\ell}) < x_\ell$, $\ell=1,\ldots,r$ in \eqref{eq:psy} as
\[
v_i < \min(\xIi)=:m_i,\quad  i=1,\dots,K.
\]
Then by Proposition \ref{prop:psy}, the expression in  \eqref{eq:convolution0} becomes
\equh\label{eq:convolution}
\frac{(1-\beta)^{K}}{\Gamma(\beta_p)^r}\int_{\vv 0< \vv x< \vv t}\int_{\vv 0<\vv v<\vv m}\prodd i1K\pp{g_{|\calI(i)|}^{(\beta)}(\xIi-v_i\vv 1)v_i^{-\beta}}d\vv v d\vv x.
\eque
A careful  examination shows that
\[
g_{|\calI(i)|}^{(\beta)}(\xIi-v_i\vv 1) =\frac1{\Gamma(\beta)\Gamma(2-\beta)} (m_i-v_i)^{\beta-1}h_{|\calI(i)|}^{(\beta)}(\vv x_{\calI(i)}).
\]
\comment{
Observe that $g_{|\calI(i)|}(\xIi-v_i\vv 1) = g_{|\calI(i)|}^{(\beta)}(\vv x_{\calI(i)}^*)$ with $\vv x_{\calI(i)}^* = (x^*_\ell)_{\ell\in \calI(i)}$ given by
\[
x^*_\ell := \begin{cases}
x_\ell & x_\ell \ne m_i\\
m_i-v_i & x_\ell = m_i, 
\end{cases}
\ell\in\calI(i).
\]
(We only need consider the case $m_i = \min(\xIi)$ is uniquely achieved, since $g_{|\calI(i)|}$ is only non-zero off-diagonal.) 
}Then, \eqref{eq:convolution} becomes
\begin{multline*}
\frac1{\Gamma(\beta_p)^r}\pp{\frac1{\Gamma(\beta)\Gamma(1-\beta)}}^K\\
\times \int_{\vv 0<\vv x< \vv t}\int_{\vv 0<\vv  v<\vv m}\prodd i1K\pp{(m_i-v_i)^{\beta-1}v_i^{-\beta}h_{|\calI(i)|}^{(\beta)}(\vv x_{\calI(i)})}d\vv vd \vv x\\
= \frac1{\Gamma(\beta_p)^r}\int_{\vv 0<\vv x< \vv t}\prodd i1Kh_{|\calI(i)|}^{(\beta)}(\xIi)d\vv x,
\end{multline*}
by   integrating with respect to each $v_i$ separately and applying the relation  between beta and gamma functions. 
Then the desired result follows.
\end{proof}
In particular, we have the following.
\begin{Cor}\label{Cor:incre moment}
Let $L_{I,t}$ be as in \eqref{eq:LIt}. Then for $0\le s<t\le 1$,
\equh\label{eq:4.5}
\E \left(L_{I,t}-L_{I,s}\right)^r = \E L_{I,t-s}^r =  \frac{\Gamma(\beta)^p\Gamma(2-\beta)^pr!}{ \Gamma(\beta_p) \Gamma((r-1)\beta_p+2)}   \cdot    (t-s)^{(r-1)\beta_p+1}.
\eque
\end{Cor}
\begin{proof}
The second equality follows  from \eqref{eq:moment limit}  with $I_1=
\cdots
=I_r=I$ and the following identity:
\[
\int_{0<x_1<\cdots<x_r<1}\prodd i2r(x_{i}-x_{i-1})^{\gamma}d\vv x= \frac{\Gamma(\gamma+1)^{r-1} }{\Gamma(r(\gamma+1)-\gamma+1)} \mfa \gamma>-1, r\ge 2,
\]
which can be obtained by changes of variables and the relation  between beta and gamma functions.
The first equality  can  be either derived from \eqref{eq:moment limit} through an expansion, or  from the fact that each underlying shifted $\beta$-stable regenerative set $R_i+V_i$ is stationary when restricted to the interval $[0,1]$ (Remark \ref{rem:stationary}).   
\end{proof}
\begin{Rem}
As mentioned before Remark \ref{rem:stationary},  
when restricted to $[0,1]$, $L_{I,t} \stackrel{d}{=} M_{\beta_p}((t-V_I)_+)$ where $V_I$ is a sub-random variable with  density function $c_{\beta,p}(1-\beta_p)v^{-\beta_p}$ with $c_{\beta,p} = (\Gamma(\beta)\Gamma(2-\beta))^p/(\Gamma(\beta_p)\Gamma(2-\beta_p))$ \citep[Eq.(B.9)]{samorodnitsky17extremal}. Therefore, all the properties of $(L_{I,t})_{t\in[0,1]}$, for a single fixed $I$, can also be derived from the corresponding $(M_{\beta_p}((t-V_I)_+)_{t\in[0,1]}$, where $\proba(V_{I}\le v) = v^{1-\beta_p}$ and $V_{I}$ is independent from $M_{\beta_p}$. For example, the $r$-th moments of the latter have been known \citep[bottom of page 77]{owada16limit}, and they entail \eqref{eq:4.5} as an alternative proof.
\end{Rem}
\subsection{Random covering scheme}\label{sec:covering}
To establish Proposition \ref{prop:psy}, we shall use a construction of local times
  motivated from   the so-called random covering scheme, by first constructing a stable regenerative set as the set left uncovered by a family of random open intervals based on a Poisson point process  (e.g.~\citep{bertoin00two,fitzsimmons85set} and \citep[Chapter 7]{bertoin99subordinators}).   

We shall work with a specific construction of $(R_i)_{i\in\N}$ as follows.  Let  $\cl{N}=\sum_{\ell\in\N} \delta_{(a_\ell,y_\ell,z_\ell)}$   be a Poisson point process on $[0,K) \times \R_+\times\R_+$, $K\in \N$, with intensity measure 
$dadyz^{-2}dz$, where $\delta$ denotes the Dirac measure.
 Define
\[
O_i:=\bigcup_{\ell:a_\ell\in J_i}(y_\ell,y_\ell+z_\ell),\quad R_i:= [0,\infty)\setminus O_i,\quad i=1,\ldots,K,
\]
where $J_i=[i-1,i-\beta)$.
It is known that $(R_i)_{i=1,\dots,K}$ constructed above are i.i.d.~$\beta$-stable regenerative sets starting at the origin \citep[Example 1]{fitzsimmons85set}.
In this section we shall work with deterministic shifts
\[
\vv v = (v_1,\dots,v_K)\in (0,1)^K.
\]
 Let
\begin{equation}\label{eq:D_p(m)}
\cl{D}_p(m):=\{I\in \cl{D}_p: ~\max I\le m\}, ~m\in\N.
\end{equation}
where $\cl{D}_p$ is as in \eqref{eq:D_p}. 
With  the functional $L_t$  in \eqref{eq:L_t},  consider
\equh\label{eq:LItv}
L_{I,t} \equiv L_t\pp{\bigcap_{i\in I}(R_i+v_i)},\ I\in\calD_p(K),\ t\ge 0,
\eque
where $(R_i)_{i\in\N}$ are as above. We emphasize that the  notation in \eqref{eq:LItv}   is {\em strictly restricted to this section}, and in particular is different from our notation of $L_{I,t}$ in the other sections, where $v_i$ will be replaced by random $V_i$.

Next, we consider the following approximations of $(R_i)_{i=1,\dots,K}$. For any $\epsilon>0$,  we set
\[
O_i^{(\epsilon)}:=\bigcup_{\ell:a_\ell\in J_i, z_\ell \ge \epsilon}(y_\ell,y_\ell+z_\ell),\quad R_i^{(\epsilon)}:= [0,\infty)\setminus O_i^{(\epsilon)},\quad i=1,\ldots,K.
\]
Define  
\[
\wt R_i\topp\epsilon := R_i\topp\epsilon+v_i \qmand \wt R_I\topp\epsilon := \bigcap_{i\in I}\wt R_i\topp\epsilon, 
\quad I\in\calD_p(K).
\]
Introduce then
\begin{equation}\label{eq:Delta_st}
L_{I,t}\topp\epsilon:= \frac{1}{\Gamma(\beta_p)} \left(\frac{\epsilon}{e}\right)^{{\beta_p-1}}\int_0^t \inddd{x\in   \wt{R}_I^{(\epsilon)} } dx    
\qmand \Delta_{s,t}^{(\epsilon)}(I):=L_{I,t}\topp\epsilon-L_{I,s}^{(\epsilon)},
\end{equation}
for $0<s<t$. 
 Set also
 \[\cl{N}_\epsilon:= \sum_{\ell:\, z_\ell \ge \epsilon} \delta_{(a_\ell,y_\ell,z_\ell)}.
 \]
 
Below we begin with calculating certain asymptotic moments involving \eqref{eq:Delta_st}. 

\begin{Lem}\label{lem:1}
For any $I_\ell\in \cl{D}_p(K)$, $\vv v \in (0,1)^K$, and $s_\ell, t_\ell$ satisfying $\max(\vv v_{I_\ell}) < s_\ell < t_\ell \le 1$, $\ell=1,\dots,r$,  we have
\begin{align}
&\lim_{\substack{s_\ell \downarrow \max (\vv v_{I_\ell}), \\ \ell=1,\dots,r}} \lim_{\epsilon \downarrow 0} \esp\pp{\prodd\ell1r \Delta_{s_\ell, t_\ell}^{(\epsilon)}(I_\ell) } \label{eq:psy1} \\
&\quad = \Gamma(\beta_p)^{-r}  \int_{\max(\vv v_{I_\ell}) < x_\ell < t_\ell, \, \ell=1,\dots,r} \prod_{i=1}^K  g_{|\cl{I}(i)|}^{(\beta)}(\xIi-v_{i}\vv1)d\vv x. \notag
\end{align}

\end{Lem} 
We start with a preparation. Define $g_{q,\epsilon}^{(\beta)}$ similarly as $g_q^{(\beta)}$ in \eqref{eq:g_q} as the symmetric function determined by 
\[
g_{q,\epsilon}^{(\beta)}(x_1,\ldots,x_q)  =  \prod_{j=1}^q f_\epsilon(x_j - x_{j-1}), \ \ x_0:=0 < x_1 < \cdots < x_q <1, 
\]
where 
\begin{equation}\label{eq:h_epsilon}
f_\epsilon(y):= \big(e^{y/\epsilon-1}\epsilon \big)^{\beta-1}1_{\{y\le \epsilon\}} +  y^{\beta-1}1_{\{y> \epsilon\}},\quad y>0.
\end{equation}
We   set  also $g_{0,\epsilon}^{(\beta)}:=1$. 
\begin{proof} [Proof of Lemma \ref{lem:1}] 
First, we   claim that if
\begin{equation*} 
(x_1,\ldots,x_q)\in D_q:=\{(x_1,\ldots,x_q) \in (0,1)^q:\ x_i\neq x_j \text{ for }i\neq j \},  \ q\in \N,
\end{equation*}
then for $\epsilon\in (0,1)$,
\begin{equation}   \label{e:multi.pts}
P\pp{x_i \in  {R}^{(\epsilon)}_1, ~i=1,\ldots,q} = \left(\frac{e}{\epsilon}\right)^{q(\beta-1)} g_{q,\epsilon}^{(\beta)}(x_1,\ldots,x_q).
\end{equation}
For the proof, assume without loss of generality that  $x_0=0<x_1<\cdots<x_q< 1$.
 Observe that 
  the event in the probability sign in \eqref{e:multi.pts} occurs exactly when the Poisson point process $\calN$ has no points in the following regions 
 \[
 \ccbb{(a,y,z)\in[0,1-\beta)\times[x_{i-1},x_i)\times\R_+:\, y+z>x_i,\, z>\epsilon }, \ \ i=1,\dots, q. 
 \]
 Therefore,
\[
P\pp{x_i \in  {R}_1^{(\epsilon)}, ~i=1,\ldots,q}
=\prod_{i=1}^q \exp\pp{ -(1-\beta)  \int_{x_{i-1}}^{x_i}\int_{\max\{x_i-y, \epsilon\}}^\infty     \frac{1}{z^2}  dz dy }.
\]
By elementary calculations,
\[
\int_{x_{i-1}}^{x_i}\int_{\max\{x_i-y, \epsilon\}}^\infty     \frac{1}{z^2}  dz dy = \begin{cases}
\displaystyle\frac{x_i-x_{i-1}}\epsilon & \text{ if } x_i-x_{i-1}\le \epsilon;\\ \\
\displaystyle\log\pp{\frac e\epsilon(x_i-x_{i-1})} & \text{ if } x_i-x_{i-1}>\epsilon.
\end{cases} 
\]
Putting these together yields the desired result.

Now let us turn 
our
 attention to proving \eqref{eq:psy1}. 
We have, by \eqref{eq:Delta_st} and Fubini,
\begin{align}
\E  \pp{ \prod_{\ell =1}^r \Delta_{s_\ell,t_\ell}^{(\epsilon)}(I_\ell)}
& = \frac1{\Gamma(\beta_p)^r} \left(\frac{\epsilon}{e}\right)^{rp(\beta-1)}\E\left( \prod_{\ell=1}^r\int_{s_\ell}^{t_{\ell}} \inddd{x\in   \wt{R}_{I_\ell}^{(\epsilon)}} dx\right)\nonumber\\
&= \frac1{\Gamma(\beta_p)^r}  \left(\frac{\epsilon}{e}\right)^{rp(\beta-1)} \int_{\vv s < \vv x < \vv t}
  P\pp{x_\ell\in \wt{R}_{I_\ell}^{(\epsilon)},~\ell=1,\dots,r }d\vv x. \label{eq:Delta_epsilon}
\end{align}
Notice that
$\sccbb{x_\ell\in\wt R_{I_\ell}\topp\epsilon} = \bigcap_{i:\ell\in\calI(i)}\sccbb{x_\ell-v_i\in R_i\topp\epsilon}$.
Therefore by independence,  we get
\[
P\pp{x_\ell\in \wt{R}_{I_\ell}^{(\epsilon)},~\ell=1,\ldots,r } = \prod_{i=1}^K  P\pp{x_\ell-v_i\in  {R}_i^{(\epsilon)},~ \ell\in \cl{I}(i)}.
\]
Note that the probability above is  zero if one of $x_\ell-v_i$ is negative, $i=1,\ldots,K$. Hence
by \eqref{e:multi.pts} and the fact $\sum_{i=1}^K |\cl{I}(i)|=rp$,  we have
\[
 \prod_{i=1}^K P\pp{x_\ell-v_i\in  {R}_i^{(\epsilon)},~ \ell\in \cl{I}(i)}=  
 \left(\frac e{\epsilon}\right)^{rp(\beta-1)}\prod_{i=1}^K  g_{|\cl{I}(i)|,\epsilon}^{(\beta)}(\xIi-v_{i}\vv1) 1_{\{\xIi\ge v_{i}\vv1\}}.
\]
Summing up, in view of \eqref{eq:Delta_epsilon}, we claim that
\begin{align}
&\lim_{\substack{s_\ell \downarrow \max (\vv v_{I_\ell}), \\ \ell=1,\dots,r}} \lim_{\epsilon \downarrow 0} \esp\pp{\prodd\ell1r \Delta_{s_\ell, t_\ell}^{(\epsilon)}(I_\ell) }  \label{eq:moment epsilon zero} \\ &= \Gamma(\beta_p)^{-r}   \int_{\max(\vv v_{I_\ell}) < x_\ell < t_\ell, \, \ell=1,\dots,r}\prod_{i=1}^K  g_{|\cl{I}(i)|}^{(\beta)}(\xIi-v_{i}\vv1)d\vv x  \notag
\end{align}
where $g_q^{(\beta)}$ is as in \eqref{eq:g_q}.
Indeed it is elementary to verify from \eqref{eq:h_epsilon} that as $\epsilon\downarrow 0$, we have $f_\epsilon(y) \uparrow y^{\beta-1}$ for any $y>0$, and hence $g_{q,\epsilon}^{(\beta)}\uparrow g_q^{(\beta)}$ a.e.. So \eqref{eq:moment epsilon zero} follows from    the monotone convergence theorem. 
\end{proof}

Next in order to establish Proposition \ref{prop:psy}, we need   to identify an a.s.~limit of \\ $\lim_{\Q\ni s_\ell \downarrow \max (\vv v_{I_\ell})} \lim_{\epsilon \downarrow 0} \prodd\ell1r \Delta_{s_\ell, t_\ell}^{(\epsilon)}(I_\ell)$, together with an interchangeability between the limits and an expectation. To this aim we shall provide the following two lemmas. In the first lemma below, if $p=1$, this is the same result as that in \citep{bertoin00two}. For general $I$ the proof follows the same strategy. 

\begin{Lem}
For every $I\in \cl{D}_p(K)$ and $s,t$ satisfying $\max(\vv v_I)<s<t\le 1$,
\begin{equation}\label{eq:rev mart}
\E\pp{\Delta_{s,t}^{(\eta)}(I)   \mmid\cl{N}_\epsilon}=   \Delta_{s,t}^{(\epsilon)}(I)  ~a.s.  \mfa \ 0<\eta<\epsilon< s -\max(\vv v_I).
\end{equation}
\end{Lem}
 \begin{proof}
For $\eta\in (0,\epsilon)$, define
\[
O_i^{(\eta,\epsilon)}=\bigcup_{\ell:a_\ell\in J_i, z_\ell \in [\eta, \epsilon)}(y_\ell,y_\ell+z_\ell),\quad R_i^{(\eta,\epsilon)}= [0,\infty)\setminus O_i^{(\eta,\epsilon)},\quad i=1,\ldots,K,
\]
and define 
\[
\wt R_i\topp{\eta,\epsilon}:= R_i\topp{\eta,\epsilon}+v_i \qmand  \wt R_I\topp{\eta,\epsilon}:=\bigcap_{i\in I}\wt R_i\topp{\eta,\epsilon}, \ I \in \mathcal D_p(K). 
\]   Then   for $0<\eta<\epsilon< s -\max(\vv v_I)$,  by Fubini's theorem and the independence property of  the Poisson point process, we have
\begin{align}\label{eq:check mart}
\E\pp{\int_s^t  \inddd{x\in    \wt{R}_I^{(\eta)}}   dx \mmid\cl{N}_\epsilon} & =\int_s^t \E\pp{ \inddd{x\in    \wt{R}_I^{(\epsilon)} } \inddd{x\in    \wt{R}_I^{(\eta,\epsilon)}} \mmid  \cl{N}_\epsilon}dx\notag\\
&=\int_s^t   P\pp{x\in    \wt{R}_I^{(\eta,\epsilon)} } \inddd{x\in    \wt{R}_I^{(\epsilon)} }  dx.
\end{align}
By a calculation similar  to that in the proof of Lemma \ref{lem:1} (see also \citep[page 10]{bertoin00two}),  we have,  for $w>\epsilon$, 
\begin{align*}
P\pp{w\in    {R}_i^{(\eta,\epsilon)} }
  = \exp\pp{ -(1-\beta)  \iint  1_{\{ y <w <y+z,\ z\in [\eta,\epsilon)\}}     \frac{1}{z^2}   dzdy }=\left(\frac{\eta}{\epsilon}\right)^{1-\beta}.
\end{align*}
Hence
\[
P\pp{x\in    \wt{R}_I^{(\eta,\epsilon)} }=\prod_{i\in I} P\pp{x-v_i\in    {R}_i^{(\eta,\epsilon)} }=\left(\frac{\eta}{\epsilon}\right)^{p(1-\beta )}=\left(\frac{\eta}{\epsilon}\right)^{1-\beta_p}.
\]
Plugging this back into (\ref{eq:check mart}), we obtain (\ref{eq:rev mart}).
\end{proof}
This lemma says that $(\Delta_{s,t}\topp\epsilon (I))_{\epsilon\in(0,s-\max(\vv v_I))}$ is a martingale as $\epsilon\downarrow0$ with respect to the filtration $(\sigma(\calN_\epsilon))_{\epsilon>0}$. 
 Since the convergence of the moments of $\Delta_{s,t}^{(\epsilon)}(I)$ as $\epsilon\downarrow 0$, was established in the proof of Lemma \ref{lem:1}, 
 by the martingale convergence theorem, we have for every $0 < s < t \le1$, 
\begin{equation}\label{eq:Delta conv}
\lim_{\epsilon\downarrow 0} \Delta_{s,t}^{(\epsilon)}(I) =: \Delta^*_{s,t}(I) \mbox{ a.s. and in $L^m$ for all $m\in \N$}.
\end{equation}
Then there exists a probability-one set, on which the convergence in \eqref{eq:Delta conv} holds 
     for all $s\in \mathbb{Q}\cap (0,t)$. Since $\Delta^*_{s,t}(I)$ is non-increasing in $s\in \mathbb{Q}\cap (0,t)$, one can a.s.\ define 
\begin{equation}\label{eq:L_t(I) cover}L_{I,t}^*: =\begin{cases}
 \lim\limits_{\Q\ni s\downarrow \max(\vv v_I)} \Delta^*_{s,t}(I),  &\text{ if } \max(\vv v_I)<t,\\
 0  & \text{ if }\max(\vv v_I)\ge t.
\end{cases}
\end{equation}

\begin{Lem}\label{lem:local time identical}
For any $0 < t\le 1$, $\vv {v}\in (0,1)^K$, and any $I\in\calD_p(K)$, we have
$L_{I,t} = L^*_{I,t}$ almost surely.
\end{Lem}  
\begin{proof}
First we write
\[
\wt{R}_I=\bigcap_{i\in I} (R_i+v_i)=R_I+ V_I
\] 
where
$ V_I:=\inf \wt{R}_I$ and $R_I:= (\wt{R}_I\cap [V_I,\infty))-V_I$. (Note that even with all $v_i$ fixed, $V_I$ is still a non-degenerate random variable with probability one, unless $v_i = v$ for all $i\in I$.)
In view of \citep[Lemma 3.1]{samorodnitsky17extremal},  $R_I$ is a $\beta_p$-stable regenerative set and $V_I\ge 0$ is a random shift independent of $R_I$.  Observe that $L_{I,t}=L_{I,t}^*=0$ for $t\in [0,V_{I})$, so it suffices to show $L_{I,t+V_I}= L_{I,t+V_I}^*$ for any $t\ge 0$ a.s.  By \citep[Theorem 3]{kingman73intrinsic},    $L_{I,t+V_I}=L_t(R_I)$ is a version of the standard local time of $R_I$  (or a standard $\beta_p$-Mittag--Leffler process). Here by ``standard'', we mean  that $L_t(R_I)$   has the same law as the inverse of a standard $\beta_p$-stable subordinator satisfying \eqref{eq:laplace} but with $\beta$ there replaced by $\beta_p$.
 On the other hand, 
using  Kolmogorov's   criterion \citep[Theorem 3.23]{kallenberg02foundations} and the formula of moments in Lemma \ref{lem:1} above, one can verify that  $\{L_{I,t}^*\}_{t\ge 0}$ admits a version which is continuous in $t$. It also follows from the  construction  that $L_{I,t+V_I}^*$ is additive and increases only over $t\in R _I$.   
Then by Maisonneuve \citep[Theorem 3.1]{maisonneuve87subordinators}, for some constant $c>0$,     $L_{I,t+V_I}^*= cL_{I,t+V_I}$ almost surely for each $t\ge 0$.

We shall show that $c=1$. Taking $t=1$, $\esp L_{I,1+V_I} = 1/\Gamma(\beta_p+1)$ by our knowledge of Mittag--Leffler process (e.g.~\citep[Proposition 1(a)]{bingham71limit}). Now to show $c=1$, it suffices to show that $\E L_{I,1+V_I}^* = 1/\Gamma(\beta_p+1)$.     

Let $(L_{I,t}^{o})_{t\ge 0}$ be     $(L_{I,t}^*)_{t\ge 0}$ in \eqref{eq:L_t(I) cover}  but with $\vv{v}_I=\vv 0$. From \eqref{eq:psy},  one may verify that $\E L_{I,1}^o = 1/\Gamma(\beta_p+1)$ (in fact, comparing all the moments leads to
$L_{I,1}^o \eqd L_{I,1+V_I}$).  
The proof is concluded by showing that
 \equh\label{eq:strongMarkov}
 (L_{I,t+V_I}^*)_{t\ge 0} \eqd (L_{I,t}^o)_{t\ge 0}.
 \eque
This  essentially follows from a strong regenerative  property. Indeed,
for fixed $\epsilon>0$, let 
$\cl{G}\topp\epsilon_t$, $t\ge 0$, be the augmented  filtration 
generated by the $p$-dimensional process $(D_{i,t}^{(\epsilon)},\ i\in I)_{t\ge 0}$, where $D_{i,t}\topp\epsilon=\inf ( \wt{R}_i^{(\epsilon)} \cap (t,\infty))$.  Note that for each $i\in I$, $\wt R_i\topp\epsilon = R_i\topp\epsilon+v_i$ is regenerative with respect to $(\calG\topp\epsilon_t)_{t\ge 0}$ in the sense of \citep[Definition 1.1]{fitzsimmons85intersections}: this can be seen from the fact that $R\topp\epsilon_i$ is regenerative with respect to $(\calG\topp\epsilon_{t+v_i})_{t\ge 0}$ (see e.g.~\citep[Eq.(6)]{fitzsimmons85set}).

Next, consider the shift operator $\theta_t$ on $\mathbf{F}$  as
$
\theta_t F= (F\cap [t,\infty))-t,
$
for $t\ge 0$.
Write $
V_I^{(\epsilon)}:=\inf \wt{R}_I^{(\epsilon)}$, which is finite almost surely.
Observe that $V_I\topp\epsilon = \inf\{t>0:D_{i,t-}\topp\epsilon = t, \mfa i\in I\}$, and hence it is an optional time with respect to 
$(\calG_t\topp\epsilon)_{t\ge0}$. Note in addition  that $V_I^{(\epsilon)}\in \wt{R}_i^{(\epsilon)}$ for all $i\in I$, and that $\theta_{V_I^{(\epsilon)}}\wt{R}_i^{(\epsilon)}$'s are conditionally independent given $\cl{G}\topp\epsilon_{V_I^{(\epsilon)}}$
 So it follows from the strong regenerative property (\citep[Proposition (1.4)]{fitzsimmons85intersections}) that
$
\left(\theta_{V_I^{(\epsilon)}} \wt{R}^{(\epsilon)}_i\right)_{i\in I}\overset{d}{=} \left( R_i^{(\epsilon)}\right)_{i\in I}.
$
Therefore, 
\[
\left(\int_s^t \inddd{x\in \theta_{V_I^{(\epsilon)}} \wt{R}_I^{(\epsilon)}}dx\right)_{0<s<t}  \EqD  \left(\int_s^t \inddd{x\in  R _I^{(\epsilon)}}dx\right)_{0<s<t}.
\]

Now, examining the construction starting from \eqref{eq:Delta_st}, we see that the relation above leads to \eqref{eq:strongMarkov}.
This completes the proof. 
\end{proof}

By combining all the lemmas above, it is now straightforward to complete the proof of Proposition \ref{prop:psy}. 
\begin{proof}[Proof of Proposition \ref{prop:psy}]
In view of Lemmas \ref{lem:1} and \ref{lem:local time identical}, it suffices to show that
$$
\lim_{\substack{\Q\ni s_\ell\downarrow \max \vv (v_{I_\ell})\\\ell=1,\dots,r}} \lim_{\epsilon \downarrow 0} \esp \pp{ \prod_{\ell=1}^r \Delta_{s_\ell, t_\ell}^{(\epsilon)} (I_\ell) } = \esp \pp{\prodd \ell 1rL_{I_\ell,t_\ell}^*}.
$$
By the $L^m$ convergence in \eqref{eq:Delta conv}, 
$$
\lim_{\epsilon \downarrow 0} \esp \pp{ \prod_{\ell=1}^r \Delta_{s_\ell, t_\ell}^{(\epsilon)} (I_\ell) } = \esp \pp{ \prod_{\ell=1}^r \Delta_{s_\ell, t_\ell}^*(I_\ell) }.
$$
It then remains to show that
\[
\lim_{\substack{\Q\ni s_\ell\downarrow \max \vv (v_{I_\ell})\\\ell=1,\dots,r}}\esp \pp{\prodd \ell 1r\Delta^*_{s_\ell,t_\ell}(I_\ell)} = \esp \pp{\prodd \ell 1rL_{I_\ell,t_\ell}^*},
\]
for which we have established the pointwise convergence in \eqref{eq:L_t(I) cover}. To enhance to the convergence in expectation via uniform integrability, we need a uniform upper bound  for $\esp (\prodd\ell1r \Delta_{s_\ell,t_\ell}^*(I_\ell)^2)$ in terms of $\vv s$. This follows from a  reexamination of     \eqref{eq:moment epsilon zero}.  
The proof is then completed.
\end{proof}

\section{Stable-regenerative multiple-stable processes}
\label{sec:limit process}
 \subsection{Series representations for multiple integrals}\label{sec:mult int}  
We review the the multilinear series representation of off-diagonal multiple integrals with respect to an infinitely divisible random measure without a 
Gaussian
 component. Our main reference is Szulga \citep{szulga91multiple} and Samorodnitsky \citep[Chapter 3]{samorodnitsky16stochastic}.

 Let $(E,\cl{E},\mu)$ be 
 a
 measure space where $\mu$ is $\sigma$-finite and atomless. 
First we recall the infinitely divisible random measure without Gaussian component. Let $M(\cdot)$ be such a random measure with a   control measure $\mu$. Then, its law is determined by
\[
\E e^{i\theta M(A)}=\exp\left(-\mu(A) \int_{\R  } (1-\cos(\theta y))  \rho(dy)\right),~ A\in \cl{E}, ~\mu(A)<\infty,~ \theta\in \R,
\]  
where $\rho$ is a \emph{symmetric} L\'evy measure  satisfying $\int_{\R} (1\wedge y^2) \rho(dy)\in (0,\infty)$ \citep[Section 3.2]{samorodnitsky16stochastic}.
We shall later on need a generalized inverse of the tail L\'evy measure defined as
\[
\rho^{\leftarrow}(y):=\inf\{x> 0: \rho(x,\infty)\le y/2\},\quad y>0.
\]

A special  case
of
 our interest is the symmetric $\alpha$-stable (S$\alpha$S) random measure on $(E,\calE)$, denoted by $S_\alpha$ ($\alpha\in(0,2)$), determined by
$\E e^{iu S_\alpha(A)}= \exp(-|u|^\alpha \mu(A))$ for all $A\in \cl{E}, \mu(A)<\infty$. 
In this case, the L\'evy measure is
\begin{equation}\label{eq:C_alpha}
\rho(dy)= \frac{\alpha C_\alpha}{2} |y|^{-\alpha-1}1_{\{y\neq 0\}}dy \qmwith
C_\alpha=\left(\int_0^\infty \sin(y) y^{-\alpha} dy\right)^{-1},
\end{equation}
and
$\rho^{\leftarrow}(y)=C_\alpha^{1/\alpha} y^{-1/\alpha}$, $y>0$.
Throughout we shall work with the following assumption for $\rho$: 
\equh\label{eq:rho}
\rho((x,\infty))\in \RV_\infty(-\alpha), \alpha\in(0,2)  \text{ and } \rho((x,\infty)) = O(x^{-\alpha_0}),\ \alpha_0<2 \mmas x\downarrow0,
\eque
where $\RV_\infty(-\alpha)$ denotes the class of functions regularly varying with index $-\alpha$ at infinity (\cite{bingham87regular}).

Now we introduce the series representations for   multiple integrals with respect to $M$. 
When working with series representations,
we
 shall always  treat integrands supported within a finite-measure   subspace of $E^p$. In particular,  fix an index set $\T$ 
and suppose $(f_t)_{t\in \T}$ is a family of  product measurable symmetric functions from $E^p$ to $\R$, such that $\cup_{t\in \T} \supp(f_t) \subset B^p$ for some $B\in\calE$ with $\mu(B)\in(0,\infty)$, where $\supp(f_t):=\{x\in E^p:~f_t(x)\neq 0\}$.

Now let $(\varepsilon_i)_{i\in\N}$ be i.i.d.~Rademacher random variables, $(\Gamma_i)_{i\in\N}$ be consecutive arrival times of a standard Poisson process, and  $( U_i)_{i\in\N}$ be i.i.d.~random elements taking 
values
 in $E$ with   distribution $\mu(\cdot \cap B)/\mu(B)$,  all assumed to be independent. 
Then for every $A\in\cl{E}$ with $A\subset B$, the series $
M_0(A):=\sum_{i=1}^\infty \varepsilon_i \rho^{\leftarrow}(\Gamma_i/\mu(B)) \delta_{U_i}(A) 
$
converges a.s.\ and $M_0\EqD M$ (\citep[Theorem 3.4.3]{samorodnitsky16stochastic}, see also \citep{rosinski99product}). Without loss of generality  we shall make the identification $M=M_0$.  Then the (off-diagonal) multiple integral of $f_t$ with respect to $M$ can be defined as
\begin{align} \label{eq:series rep}
&\left(\int_{B^p}'  f_t(x_1,\ldots,x_p) M(dx_1)\cdots M(dx_p)\right)_{t\in \T}\\ \notag   
&=\left( p!\sum_{ I \in \cl{D}_p} \left(\prod_{i\in I}\varepsilon_i  \rho^{\leftarrow}(\Gamma_i/\mu(B))\right)  f_t(\vv U_I)\right)_{t\in \T},
\end{align}
where  
\[
\vv U_{I}\equiv(U_{i_1},\ldots,U_{i_p}) \mbox{ for } I =(i_1,\dots,i_p)\in \cl D_p,
\]
as long as the   multilinear series in \eqref{eq:series rep} converges a.s. It is known that  the convergence   holds if and only if
\begin{equation}\label{eq:series finite}
\sum_{ I\in\cl{D}_p }  \prod_{i\in I}   \rho^{\leftarrow}(\Gamma_i/\mu(B)))^2  f_t(\vv U_I)^2<\infty \quad \mbox{ a.s.,}
\end{equation} 
and in this case the convergence also holds unconditionally, namely, regardless of any deterministic permutation of its entries (\citep{kwapien92random} 
 and \citep[Remark 1.5]{samorodnitsky89asymptotic}). 
On the other hand, a non-symmetric integrand, say $g$, can always be symmetrized without affecting the resulting multiple stochastic integral, by considering $ (p!)\inv\sum_\sigma g(x_{\sigma(1)},\dots,x_{\sigma(p)})$, summing over all permutations of $\{1,\dots,p\}$.  

The following lemma provides a condition to verify the convergence under \eqref{eq:rho}.
\begin{Lem}\label{lem:A1} Let $(\varepsilon_i)_{i\in\N}$ and $(\Gamma_i)_{i\in\N}$ be as 
above and let $f:E^p\rightarrow \R$ be a measurable symmetric function.
 For every $p\in\N, c>0$, 
\[
\sum_{ I\in\cl{D}_p } \left(\prod_{i\in I}\varepsilon_i  \rho^\leftarrow(\Gamma_i/c)\right)  f(\vv U_I) 
\]
converges almost surely and unconditionally, if 
  $\esp f(\vv U_I)^2<\infty$. 
\end{Lem}
\begin{proof}
It suffices to prove for $c=1$, and in this case the convergence criterion \eqref{eq:series finite} becomes
\begin{equation}\label{eq:series finite lemma}
\sum_{ I\in\cl{D}_p }  \prod_{i\in I}  \rho^\leftarrow(\Gamma_i)^2  f(\vv U_I)^2<\infty \quad 
\mbox{ a.s.}
\end{equation} 
Define
\begin{align}\label{eq:D le p M}
\cl{D}_{\le p}(M) & :=\{I\in \calD_k: \ 0\le k\le p,\ \max I\le M\},\\
\label{eq:H(k,M)}
\cl{H}(k,M) & :=\{I\in \cl{D}_k:\ \min I> M\},~ k=0,\dots,p,
\end{align}
for $M\in\N$, to be chosen later, where $\calD_k$ is as in \eqref{eq:D_p} with  $\calD_0=\emptyset$.
Then the series in (\ref{eq:series finite lemma}) is equal to
\begin{align}\label{eq:series low+high}
 \sum_{I_1\in  \cl{D}_{\le p}(M)}   \left( \prod_{i\in I_1} \rho^\leftarrow(\Gamma_i)^2\right) \left[ \sum_{I_2\in \cl{H}(p-|I_1|,M)}\left( \prod_{i\in I_2} \rho^\leftarrow(\Gamma_i)^2 \right)f(\vv U_{I_1\cup I_2})^2\right]. 
\end{align}
 Note that $\cl{D}_{\le p}(M)$ is finite. Hence
to prove the almost-sure convergence of the non-negative series, it suffices to show   that for each $I_1\in \cl{D}_{\le p}(M)$, the term in the bracket of (\ref{eq:series low+high})  is finite almost surely. 
This follows, in view of \eqref{eq:series finite}, if  we can show that
\[ \sum_{I_2\in \cl{H}(k,M)}  \E\left(\prod_{i\in I_2}  \rho^\leftarrow(\Gamma_i)^2\right) \E f(\vv U_{I_1\cup I_2})^2  <\infty,~k=1,\ldots,p.
\]
From assumption \eqref{eq:rho}, it follows that $\rho^\leftarrow(x)\in\RV_0(-1/\alpha)$, where the latter denotes the class of functions regularly varying at zero, and $\rho^\leftarrow(x)= O(x^{-1/\alpha_0})$ as $x\to\infty$. By Potter's bound and the fact that $\rho^\leftarrow$ is monotone, it then follows that there exists $C>0$ and $\epsilon>0$ such that
\[
\rho^\leftarrow(x)\le C\pp{x^{-1/\alpha_0}+x^{-(1/\alpha)-\epsilon}}, \mfa x>0.
\]
The following estimate  can be obtained via H\"older's inequality    as in \citep[Eq.(3.2)]{samorodnitsky89asymptotic}:
given $\delta>0$, there exists a constant $C>0$, such that  
\begin{equation}\label{eq:gamma estimate}
\E \pp{\prod_{i\in I_2}\Gamma_i^{-\delta} }\le  C  \prod_{i\in I_2}i^{-\delta} \quad \mfa I_2 = (i_1,\dots,i_k)\in \calD_k \mwith i_1>\delta k.
\end{equation} 
It then follows that for all $\delta_1,\delta_2>0$,
\[
\esp\pp{\prod_{i\in I_2}(\Gamma_i^{-\delta_1}+\Gamma_i^{-\delta_2})}\le C\prod_{i\in I_2} i^{-(\delta_1\wedge \delta_2)}  ~
\mbox{ for all $I_2\in \calD_k$ s.t.} \min I_2>(\delta_1\vee \delta_2)k.
\]
Therefore, taking $M>2p\max\{1/\alpha_0,(1/\alpha +\epsilon)\}$ and $\alpha^*:=((1/\alpha)+\epsilon)\wedge1/\alpha_0>1/2$ we have, 
\begin{align*}
\sum_{I\in \cl{H}(k,M)} \E\pp{ \prod_{i\in I} \rho^\leftarrow(\Gamma_i)^2}
&\le  C\sum_{I\in \calH(k,M)}\prod_{i\in I } i^{-2\alpha^*} \\
& \le C\sum_{I\in \calD_k}\prod_{i\in I}i^{-2\alpha^*} 
\le C\left(\sum_{i=1}^\infty i^{-2\alpha^*} \right)^k<\infty.\nonumber
\end{align*}
\end{proof}

\subsection{Stable-regenerative multiple-stable process}\label{sec:limit process def}
Recall our assumption on $p,\beta$ and $\beta_p$ in \eqref{eq:beta}, and the local-time functional $L_t$ in \eqref{eq:L_t}. 
We introduce the {\em stable-regenerative multiple-stable process  of multiplicity $p$}, denoted throughout by $\Zab \equiv (\Zab(t))_{t\ge 0}$, $\alpha\in(0,2)$, via the multiple integrals:
 \equh\label{eq:Zt}
Z_{\alpha,\beta,p}(t):=\int_{(\vF \times [0,\infty))^p }' L_t\pp{\bigcap_{i=1}^p (R_i+v_i)}  S_{\alpha,\beta}(d R_1,dv_1)\cdots  S_{\alpha,\beta}(dR_p,dv_p),\ t\ge 0.
\eque
where $S_{\alpha,\beta}(\cdot)$ is a  S$\alpha$S random measure on $ \vF \times [0,\infty) $  with control measure 
$P_\beta\times (1-\beta) v^{-\beta}  dv$.   {Note that when $p=1$, the process $\Zab$ is represented as a stable integral, and in particular, is the same process known as the {\em $\beta$-Mittag--Leffler fractional S$\alpha$S motion} introduced in \citep{owada15functional}. 
The well-definedness of the  multiple integral above when $t\in [0,1]$ directly follows from Lemma \ref{lem:A1} and Theorem \ref{thm:1}, and can be similarly verified for $t>1$ by a proper scaling. 
More specifically, if $t\in[0,1]$, using the fact that  $L_t$ vanishes when any $v_i>1$ in \eqref{eq:Zt}, the process $ \Zab(t) $ can be represented in the form of \eqref{eq:Zt}, with $\vF\times[0,\infty)$ replaced by $\vF\times[0,1]$, and the control measure replaced by 
a probability measure $P_\beta\times  (1-\beta)   v^{-\beta}\inddd{v\in[0,1]}dv$.
Then,  as in \eqref{eq:series rep}, one can obtain the series representation
\begin{equation}\label{eq:Z_t [0,1]}
\left(\Zab(t)\right)_{t\in [0,1]}
\overset{f.d.d.}{=}\left(p!  C_\alpha^{p/\alpha} \sum_{I\in \cl{D}_p}\left(\prod_{i\in I}  \varepsilon_i \Gamma_i^{-1/\alpha } \right) L_t\left(\bigcap_{i\in I}  (R_i+V_i)\right)\right)_{t\in [0,1]},
\end{equation}
where $f.d.d.$ stands for finite-dimensional distributions, $C_\alpha$  is as in \eqref{eq:C_alpha}, $(\varepsilon_i)_{i\in\N}, (\Gamma_i)_{i\in\N}$   are as in Section  \ref{sec:mult int},
 $(R_i)_{i\in\N}$ are i.i.d.~$\beta$-stable regenerative sets, $(V_i)_{i\in\N}$ are i.i.d.~random variables with law \eqref{eq:V}, and the four sequences are independent from each other.   

 As a direct consequence of the functional limit theorem proved in Theorem \ref{Thm:CLT} below and   Lamperti's theorem \citep{lamperti62semi}, the process $\Zab$ turns out to be self-similar with Hurst index 
\[
H= \beta_p+\frac{1-\beta_p}\alpha =p\pp{\frac1\alpha-1}(1-\beta)+1 \in (1/2,\infty),
\] 
that is, 
\[
(\Zab(ct))_{t\ge 0} \eqd c^H( \Zab(t))_{t\ge 0} \mfa c>0,
\]
and have stationary increments.   
In view of self-similarity, we shall only work with $(\Zab(t))_{t\in[0,1]}$   onward.  

 We conclude this section with a result on the path regularity of $\Zab$. 

 \begin{Pro}
 The process $\Zab$  admits a continuous version whose path is locally $\delta$-H\"older continuous a.s.\  for any $\delta\in (0,\beta_p)$. 
\end{Pro}
\begin{proof}
We restrict $t\in [0,1]$ without loss of generality and work with the series representation \eqref{eq:Z_t [0,1]}.
  In view of independence,
assume for convenience that the underlying probability space is the product space of $(\Omega_i,\filF_i, P_i), i=1,2$, where $(\epsilon_i)_{i\in\N}$  depends  only on $\omega_1\in \Omega_1$  and $(\Gamma_i, R_i, V_i)_{i\in\N}$ depends  only on $\omega_2\in \Omega_2$. The probability measures $P_1$ and $P_2$ are  such that those random variables   have the desired law, and $P$ is the product measure of $P_1$ and $P_2$ on the product space.   We also write $\E_i$ the integration with respect to $P_i$ over $\Omega_i$, $i=1,2$, 

We shall work with the  series representation   in \eqref{eq:Z_t [0,1]}, where without loss of generality we replace $\overset{f.d.d.}{=}$ with $=$. Then as before, write
$L_{I,t} = L_t(\bigcap_{i\in I}  (R_i+V_i))$. Since $L_{I,t}(\omega_1,\omega_2)$ is a constant function of $\omega_1$ with $\omega_2,I,t$ fixed,   we write $L_{I,t}(\omega_1,\omega_2) = L_{I,t}(\omega_2)$ for the sake of simplicity.  In addition,  we shall identify $L_{I,t}$ with its continuous version, which exists in view of  Corollary \ref{Cor:incre moment} and Kolmogorov's criterion. 

 Using a generalized Khinchine inequality for multilinear forms in Rademacher random variables (\citep{krakowiak86random}, see also \citep[Theorem 1.3 (ii)]{samorodnitsky89asymptotic}),  for any $r>1$ and some constant $C>0$, we have for $0\le s<t\le 1$ that, writing $\vv\omega = (\omega_1,\omega_2)$,
\[
\E_1 |\Zab(t)(\vv\omega)-\Zab(s)(\vv\omega)|^r \le CY_{s,t}(\omega_2)
\]
with 
\[
Y_{s,t}(\omega_2) := \pp{\sum_{I\in\calD_p} \left(\prod_{i\in I}   \Gamma_i(\omega_2)^{-2/\alpha } \right) \left|L_{I,t}(\omega_2)-L_{I,s}(\omega_2)\right|^2}^{r/2}, 0\le s<t\le 1.
\]
The two-parameter process $(Y_{s,t})_{0\le s<t\le 1}$ is finite $P_2$-almost surely in view of Lemma  \ref{lem:A1} and Corollary \ref{Cor:incre moment} (note that \eqref{eq:rho} is satisfied with $\alpha
=\alpha_0$ in this case).
Since $L_{I,t}$ is     a shifted $\beta_p$-Mittag--Leffer process,  in view of  \citep[Lemma 3.4]{owada15functional},  the random variable
\[
K_I(\omega_2):=\sup_{(s,t)\in D} \frac{|L_{I,t}(\omega_2)-L_{I,s}(\omega_2)|}{(t-s)^{\beta_p}|\log(t-s)|^{1-\beta_p}}
\]
is $P_2$-a.s.\ finite, and has finite moments of all orders,  where $D=\{(s,t): \ 0\le s<t\le 1,\ t-s<1/2 \}$.
Hence for all $(s,t)\in D$, we have
\[
\E_1 |\Zab(t)(\vv\omega)-\Zab(s)(\vv\omega)|^r\le  C(t-s)^{r\beta_p}|\log(t-s)|^{r(1-\beta_p)}  M(\omega_2),
\]
where
\[
M(\omega_2) = \pp{\sum_{I\in\calD_p} \left(\prod_{i\in I}   \Gamma_i(\omega_2)^{-2/\alpha } \right) K_I(\omega_2)^2}^{r/2},
\]
which is finite $P_2$-a.s.: this is a special case of \eqref{eq:series finite lemma}, addressed in the proof of Lemma \ref{lem:A1}.
Take $r$ large enough so that  $r\beta_p>1$.   Then by   Kolmogorov's criterion, for any $\delta\in(0,\beta_p)$ and  $P_2$-a.e.\ $\omega_2\in \Omega_2$,  $\Zab(t)(\cdot,\omega_2)$ admits a version     $\Zab^*(t)(\cdot,\omega_2)$  under $P_1$  whose path is locally $\delta$-H\"older continuous $P_1$-a.s. By Fubini, $\Zab^*(t)(\vv \omega)$ is also a version of   $\Zab(t)(\vv\omega)$  under $P_1\times P_2$ which has a locally $\delta$-H\"older continuous path $(P_1\times P_2)$-a.s.
\end{proof}
\section{A functional non-central limit theorem}\label{sec:nCLT}
\subsection{Infinite ergodic theory and Krickeberg's setup}\label{sec:ergodic}
We shall introduce some concepts in the infinite ergodic theory necessary for the formulation of our   results.   
Our main reference is Aaronson \citep{aaronson97introduction}.
Let $(E,\cl{E},\mu)$ be a measure space where $\mu$ is a $\sigma$-finite measure satisfying $\mu(E)=\infty$. Suppose that $T: E\rightarrow E$ is a measure-preserving transform, namely, $T$ is measurable and $\mu(T^{-1}B)=\mu(B)$ for all $B\in \cl{E}$. 
Let $\what T$ denote the {\em dual} (a.k.a.~Perron--Frobenius, or transfer) operator of $T$, defined by
\[
\wh{T}: L^1(\mu)\rightarrow L^1(\mu),\quad  \wh{T} g := \frac{d\mu_g  \circ T^{-1} }{d\mu},
\]
where $\mu_g(B)=\int_B g d\mu$, $B\in \cl{E}$. It is also characterized by the relation  
\begin{equation}\label{eq:dual}
\int_E (\wh{T} g) \cdot h d\mu = \int_E    g \cdot  (h\circ T) d\mu, \quad \mfa g\in L^1(\mu),~h\in L^\infty(\mu).
\end{equation}
We always assume that  $T$ is \emph{ergodic}, namely, $T^{-1} B= B$ mod $\mu$ implies either $\mu(B)=0$ or $\mu(B^c)=0$, and that $T$ is \emph{conservative}, namely, for any $B\in \cl{E}$ with $\mu(B)>0$, we have
$
\sum_{k=1}^\infty 1_B(T^k x)=\infty$ for a.e.~$x\in B$.
It is known that $T$ is ergodic and conservative,  if and only if  for any  $B\in \cl{E}$   with $\mu(B)>0$, we have
\[
\sum_{k=1}^\infty 1_{B}(T^k x)=\infty  \quad \text{ for a.e. }x\in E,
\]or equivalently
\begin{equation}\label{eq:erg and cons dual}
\sum_{k=1}^\infty \wh{T}^k g =\infty  \quad \mbox{ a.e. for all } g\in L^1(\mu),~ g\ge 0,
\mbox{ a.e.~and } \mu(g)>0.
\end{equation}
We shall, however, need a more quantitative description of the ergodic property of $T$, which provides information about the rate of divergence in \eqref{eq:erg and cons dual}.  The following assumption is formulated in the spirits of Krickeberg \citep{krickeberg67strong} and Kesseb\"ohmer and Slassi \citep{kessebohmer07limit}. We shall use the following convention throughout: \emph{any function defined on a subspace 
(e.g.~$A$)  will be  extended  to the full space (e.g.~$E$) by assuming zero value outside the subspace, whenever  necessary}. 
 
\begin{assump}\label{assump}
There exists  $A\in\cl{E}$ with $\mu(A)\in (0,\infty)$  and  $A$ is a Polish space  with $\cl{E}_A:=\cl{E}\cap A$ being its  Borel $\sigma$-field.  In addition,  there exists a positive  rate sequence  $  (b_n)_{n\in\N}$ satisfying  
\equh\label{eq:bn_RV}
(b_n)\in \RV_{\infty}(1-\beta),\quad \beta\in (0,1),
\eque
where $\RV_\infty(1-\beta)$ denotes the class of sequences regularly varying with index $1-\beta$ at infinity (\cite{bingham87regular}),
so that
\begin{equation}\label{eq:uniform ret}
\limn b_n \wh{T}^n  g(x)=  \mu(g) \quad \text{uniformly for a.e. }x\in A 
\end{equation}
for  
 all  bounded  and $\mu$-a.e.\ continuous $g$   on $A$.
\end{assump}
\begin{Rem}
The relation \eqref{eq:uniform ret} was first explicitly formulated in \cite{kessebohmer07limit} and termed as the \emph{uniform return} condition. Due to the existence of weakly wandering sets (\cite{hajian64weakly}), the   relation \eqref{eq:uniform ret}  can fail even for a  bounded integrable function $g$ supported within $A$. To be able to treat a large family of integrands $f$ in Theorem \ref{Thm:CLT} below,  we adopt an idea of \cite{krickeberg67strong}: we impose  a topological structure on the subspace $A$, and  retrain our attention to bounded and a.e.\ continuous functions supported within $A$.  It is worth noting  the resemblance of this approach to the theory of weak convergence of measures.    See Section \ref{sec:Eg} below for examples satisfying Assumption \ref{assump}.
\end{Rem}
\begin{Rem} 
Assumption \ref{assump} has an   alternative characterization in Proposition \ref{Pro:unif ret} below.  Typically, the whole space $E$ is   Polish as well.  Nevertheless, we   stress that  when   a topological concept  such as continuity, interior or boundary is mentioned, we solely refer to  the Polish topology  on the subspace $A$ (or $A^p$   in the context of product space). 
\end{Rem}

Additionally, for $A$ in Assumption \ref{assump},  and $x\in E$, we define the \emph{first entrance time}  
\begin{equation}\label{eq:phi}
\varphi(x)=\varphi_A(x)=\inf\{k\ge 1:\ T^k x\in A \},
\end{equation}
and the \emph{wandering rate} sequence 
\begin{equation}\label{eq:w_n}
w_n=\mu(\varphi \le n)=\mu\left(\bigcup_{k=1}^n 
T^{-k}
 A\right), \ \ n\in\N,
\end{equation} 
which measures  the amount of $E$ which visits $A$ up to time $n$. 
Kesseb\"ohmer and Slassi \citep[Proposition 3.1]{kessebohmer07limit} 
proved
 that under Assumption \ref{assump},
\begin{equation}\label{eq:b_n w_n}
b_n\sim \Gamma(\beta)\Gamma(2-\beta) w_n
\end{equation}
as $n\rightarrow\infty$. In particular, $w_n\in \RV_{\infty}(1-\beta)$ (note that their $\beta$ corresponds to our $1-\beta$, 
 and their $w_n$ corresponds to our $w_{n+1}$).

\subsection{A non-central limit theorem}
Let $(E,\calE,\mu)$ be $\sigma$-finite infinite measure space and $T$ a measure-preserving ergodic and conservative transform. We recall   our model, a stationary sequence $(X_k)_{n\in\N}$  in \eqref{eq:2},
where $M$ is the infinitely divisible random measure on $(E,\calE)$ with symmetric L\'evy measure $\rho$ and control measure $\mu$ as in Section \ref{sec:mult int}.

We are now ready to state the main result of the paper. Below $\mu^{\otimes p}$ denotes the $p$-product measure of $\mu$ on the product $\sigma$-field $\cl{E}^p$.
\begin{Thm}\label{Thm:CLT}
Assume $\beta, p$ and $\beta_p$ are as in \eqref{eq:beta}. For $(X_k)_{k\in\N}$ introduced in \eqref{eq:2}, suppose the following assumptions hold:
\begin{enumerate}[(a)] 
\item The   L\'evy measure $\rho$   satisfies  \eqref{eq:rho}.
\item There exists $A\in \cl{E}$  satisfying Assumption \ref{assump}, 
and $f$ is a bounded $\mu^{\otimes p}$-a.e.\ continuous function on $A^p$. 
\end{enumerate}
Then the stationary process $(X_k)_{k\in\N}$ in \eqref{eq:2} is well-defined. Furthermore, 
\begin{equation}\label{eq:nclt}
\pp{ \frac{1}{ c_n}    \sum_{k=1}^{\lf nt \rf} X_k}_{t\in [0,1]}\Rightarrow   \Gamma(\beta_p)  C_\alpha^{-p/\alpha} \mu^{\otimes p}(f)\cdot \pp{ \Zab(t)}_{t\in[0,1]},
\end{equation}
in $D([0,1])$ with respect to the uniform metric as $n\to\infty$, where  $\Zab(t)$ is the stable-regenerative multiple-stable process defined in \eqref{eq:Zt}. Moreover,
\equh
c_n = n \cdot \pp{\frac{\rho^{ \leftarrow}(1/w_n)}{ b_n}}^p \in \RV_\infty\pp{\beta_p+\frac{1-\beta_p}\alpha},\quad   \label{eq:c_n}
\eque
where  $(w_n)$ is the wandering rate associated to $A$ in \eqref{eq:w_n} and $C_\alpha$ is as in \eqref{eq:C_alpha}.
\end{Thm}
The proof of Theorem \ref{Thm:CLT} is carried out in Section \ref{sec:proof}. 
 \begin{Rem}\label{Rem:diff}
Compared to the result for $p=1$ established in \citep{owada15functional}, we assume the same assumption on $\rho$, but strictly stronger assumptions on the dynamical system and $f$.   
  Indeed, weaker notions    \emph{Darling--Kac set}
 and \emph{uniform set} were   adopted in \citep{owada15functional} instead of \eqref{eq:uniform ret}. For example, a set $A$ is a {\em Darling--Kac set} if for some positive sequence $ (a_n)_{n\in\N}$ tending to $\infty$,
 \equh\label{eq:uniform}
\frac{1}{a_n}\sum_{k=1}^n\wh{T}^k 1_A \to \mu(A)\quad \mbox{ uniformly a.e.~on $A$},
 \eque 
which is a  Ces\'aro average version of \eqref{eq:uniform ret} when $g=1_A$.   See \citep{kessebohmer07limit} for more discussions on the difference between uniform sets and uniformly returning sets. 
Also if $p = 1$, topologizing $A$ as a Polish space is unnecessary since one can apply the powerful Hopf's ratio ergodic theorem in order to treat a general $f$ (see the proof of Theorem 6.1 of \cite{owada15functional}).
 The reason that we enforce 
 a stronger assumption
 here is that for multiple integrals with $p\ge 2$, it is no longer clear how to write the statistic of interest in terms of a partial sum to which we can apply \eqref{eq:uniform}  (compare e.g.~\eqref{eq:DCT1}
  below with \citep[Eq.\ (6.10)]{owada15functional}).    It is unclear to us whether Theorem \ref{Thm:CLT} continues to hold if  Assumption \ref{assump} is relaxed to the Ces\'aro average version as in  \eqref{eq:uniform}  or even to those in \cite{owada15functional}. Nevertheless, Assumption \ref{assump}   allows us to treat a sufficiently rich class of dynamical systems and functions $f$  as exemplified in  Section \ref{sec:Eg} below. 
\end{Rem}

\subsection{Examples}\label{sec:Eg}
 We shall provide two classes of examples regarding the  assumptions involved in the main result Theorem \ref{Thm:CLT}, one about transforms on the interval $[0,1]$, and  the other   about   Markov chains. 
\begin{Eg}
The following example  can be found in Thaler \citep{thaler00asymptotics}. 
Let $(E,\cl{E})=([0,1],\cl{B}[0,1])$. 
Define  a measure by
\begin{equation*} 
\mu_q(dx)= \left(\frac{1}{x^q}+\frac{1}{(1+x)^q}\right)  1_{(0,1]}(x)dx, \quad q>1.
\end{equation*}
Define   the transformation $T=T_q: E\rightarrow E$ by
\[
T_q(x):=x\left(1+\left(\frac{x}{1+x}\right)^{q-1}-x^{q-1}\right)^{1/(1-q)} ~~(\text{mod }1).
\]
The transform $T_q$ has   an indifferent fixed point at $x=0$, namely, $T_q(0)=0$ and $T_q'(0+)=1$, and the measure  $\mu_q$  is infinite on any neighborhood of $x=0$.
Furthermore,  $T_q$ can be verified to be $\mu_q$-preserving, conservative and ergodic.   
  
 If we choose $A=[\epsilon,1]$, $\epsilon\in (0,1)$, then 
according to Thaler \citep{thaler00asymptotics}, any   Riemann integrable function  on $A$   satisfies \eqref{eq:uniform ret} and \eqref{eq:bn_RV} with $\beta=1/q$. In Theorem \ref{Thm:CLT}, we can take the $p$-variate function $f$ to be any Riemman integrable function 
with support in 
$A^p$.

In fact, the example above belongs to the so-called AFN-systems, a well-known class of interval maps possessing indifferent fixed points and an infinite invariant measure. See Zweim\"uler \citep{zweimuller98ergodic,zweimuller00ergodic} for the   definitions.  Recently for a large class of AFN-systems, Melbourne and Terhesiu \citep[Theorem 1.1]{melbourne12operator} and Gou\"ezel \citep{gouezel11correlation} established    the uniform return relation (\ref{eq:uniform ret}) with \eqref{eq:bn_RV}  for  Riemann integrable $g$    on  $A\subset [0,1]$  where $A$ is  a union of closed intervals which are away from the indifferent fixed points of $T$.

\end{Eg}

We state a primitive characterization of Assumption \ref{assump} which facilitates the discussion of the next example.
\begin{Pro}\label{Pro:unif ret}
Let $(A,\cl{E}_A)$ be as in Assumption \ref{assump}.    Assumption \ref{assump} holds if and only if  
 there exists a collection $\mathcal{C}\subset\cl{E}_A$  with the following properties:
\begin{enumerate}[(a)]
\item \label{a}$\cl{C}$ is a $\pi$-system 
containing $A$;
\item $\cl{C}$ generates the Polish topology of $A$  in the sense that for any open $G\subset A$ and any $x\in G$, there exists $ U\in \cl{C}$ such that $x\in \mathring{U} \subset U\subset G$;
\item \label{c} Any  set in $\cl{C}$ is   $\mu$-continuous;
\item \label{d} There exists a positive sequence $b_n \in \text{RV}_\infty(1-\beta)$, $0 < \beta < 1$, such that for any $B\in \cl{C}$,
\begin{equation}\label{eq:unif ret C}
b_n \wh{T}^n 1_B(x)  \rightarrow \mu(B) \quad \text{uniformly for a.e. }x\in A.
\end{equation}

\end{enumerate}
\end{Pro}  
The proof of the proposition can be found in Section \ref{sec:approx} below.

\begin{Eg}
Let $S$ be a countably infinite state space.
Consider an aperiodic irreducible and null-recurrent  Markov chain $(Y_k)_{k\ge0}$ on $S$, which has $n$-step transition probabilities   $(p^{(n)}(i,j))_{i,j\in S}$ and an invariant measure $\pi$ on $S$ which satisfies $\pi_i>0$ for any $i\in S$. Fix a state $o\in S$ and assume without loss of generality a normalization condition:
\[
\pi_o=1.
\]
Consider the path space    $E=\{x=(x(0),x(1),x(2),\ldots):~ x(k)\in S \}$  and  let $\cl{E}$ be the cylindrical $\sigma$-field.  Then one can define a    $\sigma$-finite infinite measure   $\mu$ on $(E,\cl{E})$ as
\[
\mu(\cdot )=\sum_{i\in S}\pi_i P^i(\cdot),
\]
where $P^i(\cdot)$ denotes the law $(Y_k)_{k\ge0}$ starting at state $i\in S$ at time $k=0$.  
Consider the measure preserving map of the left-shift
\[
T:E\rightarrow E ,  ~T(x(0),x(1),x(2),\ldots)=(x(1),x(2),\ldots).\]  Due to the assumptions on the chain,   the map $T$ is ergodic and conservative \citep{harris53ergodic}, and  each $P^i$ can be verified to be atomless and thus so is  $\mu$.

Now let  $A=\{x=(x(0),x(1),\ldots)\in E: x(0)=o\}$. Consider the discrete topology on $S$ induced by the metric $d(i,j)=1_{\{i\neq j\}}$, $i,j\in S$. Then the    product  space $A$ is known to be  Polish with Borel $\sigma$-field $\cl{E}_A:=\cl{E}\cap A$,   and  a  topological basis of $A$ is formed by  
\[
\cl{C}=\big\{\{x\in E: ~x(0)=o,\ x(1)=s_1, \ldots,\ x(m)=s_m\}, \ m\in \N,\ s_i\in S\big\}\cup \{\emptyset,\ A\}.
\]
 See e.g. \cite{moschovakis09descriptive}, Section 1A. 
Note that every set in $\cl{C}$ is both open and closed,   so the boundary of each is empty.  Therefore conditions 
\eqref{a}--\eqref{c}
 in Proposition \ref{Pro:unif ret} hold.

By \citep[the last line of
 page 156]{aaronson97introduction},   if $B=\{x\in A:  ~x(1)=s_1,\ldots,\ x(m)=s_m\}\in \cl{C}$,   we have   for $x=(o,x(1),x(2),\ldots)\in A$ and $n>m$ that
\begin{align*} 
(\wh{T}^n 1_B)(x) 
 =    p(o ,s_1)\cdots p(s_{m-1},s_m)  p^{(n-m)}(s_m,o)=\mu(B)  p^{(n-m)}(s_m,o).
\end{align*} 
We claim that if we assume 
\begin{equation}\label{eq:RV ret prob}
p^{(n)}(o,o)  \in \mathrm{RV}_\infty(\beta-1),
\end{equation}
then condition 
\eqref{d}
 of Proposition \ref{Pro:unif ret} holds with $b_n\sim 1/p^{(n)}(o,o)$ as $n\rightarrow\infty$.  Indeed, this is the case if for any $m\in \N$ and $s\in S$, we have
\begin{equation}\label{eq:strong ratio}
\lim_n \frac{p^{(n-m)}(s,o)}{p^{(n)}(o,o)}=1.
\end{equation}
   Condition \eqref{eq:strong ratio} is essentially   the \emph{strong ratio limit property} in \cite{orey61strong}, and   as shown there,    it  is equivalent to 
\begin{equation*} 
 \lim_n \frac{p^{(n+1)}(o,o)}{p^{(n)}(o,o)}   =1.
\end{equation*}
The last line follows from \eqref{eq:RV ret prob} and \cite[Theorem 1.9.8]{bingham87regular}. 
 
In view of the topological basis $\cl{C}$, any function $f$  on $A^p$ which depends only on a finite number of coordinates of $(x_1,\ldots,x_p)\in A^p$ can be verified to be continuous. On the other hand, a bounded continuous function   on $A^p$ depending on infinitely many coordinates can be constructed, for example, as 
$
f(x_1,\ldots,x_p)=  \sum_{n=1}^\infty 2^{-n} \sum_{j=1}^p 1_{\{x_j(n)=o\}}  .
$ 
\end{Eg}
 
   \section{Proof of the non-central limit theorem}
\label{sec:proof}

  We first provide    a summary of the proof.
We prove our main Theorem \ref{Thm:CLT} here by establishing the convergence of finite-dimensional distributions and the tightness in $D([0,1])$ separately. We shall work with our series representation established in Section \ref{sec:mult int}, and proceed by decomposing it into a leading term and a remainder term.
Most of the effort is devoted to the convergence of the finite-dimensional distributions of the leading term. For this purpose, the key is Theorem \ref{Thm:local time}  which concerns a   convergence to the joint local-time processes introduced in Section \ref{sec:local times}. To prove Theorem \ref{Thm:local time}, we shall apply the method of moments and make use of the moment formulas  established in Theorem \ref{thm:1} for the joint local-time processes.  To facilitate the moment computation,  a delicate approximation scheme is developed in Section \ref{sec:approx}. The tightness in $D([0,1])$ is also established with the aid of the aforementioned decomposition.
 Finally we note that our proof techniques are essentially different from those in the case $p=1$  considered in \cite{owada15functional}. In the case $p=1$, the proof in \cite{owada15functional} relied heavily on the infinitely divisibility of the single stochastic integral and Hopf's ratio ergodic theorem. These ingredients are non-applicable for $p\ge 2$, and our proof strategy, instead,  exploits the series representation of multiple stochastic integrals.

We now start  by a series representation of the joint distribution of $(X_k)_{k=1,\dots,n}$. For each  fixed $n\in \N$, let $(U_i^{(n)})_{i\in\N}$ be i.i.d.~taking values in $E$ following the law
\begin{equation}\label{eq:mu_n}
\mu_n(\cdot):= \frac{\mu( \cdot \cap \{\varphi\le n\}  )}{\mu(\varphi\le n)}=\frac{\mu( \cdot \cap \{\varphi\le n\}  )}{w_n},
\end{equation}
where $\varphi$ is the first entrance time to $A$ as in (\ref{eq:phi}).  
Let 
\[
T_p:=T\times\cdots \times T: E^p\to E^p
\]
 be the product transform. For each fixed $n\in \mathbb{N}$,
we apply the series representation (\ref{eq:series rep}) with $B=\{\varphi\le n\}$,
and obtain
\begin{equation}\label{eq:X_k series}
(X_k)_{k=1,\dots,n}\EqD \left(p!\sum_{ I \in \cl{D}_p} \left(\prod_{i\in I}\varepsilon_i  \rho^{\leftarrow}(\Gamma_i/w_n)\right)  f\circ T_p^k(\vv U^{(n)}_I)\right)_{k=1,\dots,n}, n\in\N,
\end{equation}
where $w_n=\mu(\varphi \le n)$ is the wandering rate sequence as in (\ref{eq:w_n}), and  $(\varepsilon_i)_{i\in\N}, (\Gamma_i)_{i\in\N}$ are as in Section  \ref{sec:mult int} and are independent from $(U_i\topp n)_{i\in\N}$. Recall the notation $\vv U_I\topp n = (U_{i_1}\topp n,\dots,U_{i_p}\topp n)$ with $I = (i_1,\dots,i_p)\in\calD_p$.
For every $n$,
the series representation converges almost surely by Lemma \ref{lem:A1}  since $f$ is bounded.

Let 
\begin{equation}\label{eq:S_n(t)}
S_n(t):= \frac{1}{c_n}\summ k1{\floor {nt}} X_k  
\end{equation}
be the normalized partial sum of interest, 
with $c_n = n (\rho^{ \leftarrow}(1/w_n)/ b_n)^p$ as in \eqref{eq:c_n}.
The proof consists of proving the convergence of finite-dimensional distributions and tightness. 

\subsection{An approximation scheme}\label{sec:approx}
Under the setup of Assumption \ref{assump}, we  introduce a  class of  functions useful for approximation purposes.  Note that the product space $A^p$ is also Polish with Borel $\sigma$-field $\cl{E}_A^p$. 
\begin{Def}\label{Def:eleme ntary}
A function $g:A^p\rightarrow \mathbb{R}$ is said to be  an \emph{elementary function}, if it is a finite linear combination of indicators of $p$-products of  $\mu$-continuity sets in $\cl{E}_A$, that is,
\[
g(x_1,\ldots,x_p)=\sum_{m=1}^M  b_m 1_{B_{1,m}\times
\cdots\times B_{p,m}}(x_1,\ldots,x_p) 
\]
where $M\in \mathbb{N}$, $b_m$'s are some real constants and  $B_{j,m}\in \cl{E}_A$ with $\mu(\partial B_{j,m} )=0$. 
A set $B\in\cl{E}_A^p$ is said to be an \emph{elementary set}, if $1_B$ is an elementary function.
\end{Def}
\begin{Lem}\label{Lem:riemann approx}
Let $f$ be a bounded $\mu^{\otimes p}$-a.e.\ continuous function on $A^p$. Then for any $\epsilon>0
$,  there exist elementary functions $g_1,g_2$ on $A^p$, such that $L(f)\le g_1\le f\le g_2\le U(f)$ and $|\mu^{\otimes p}(f)-\mu^{\otimes p}(g_i)|<\epsilon$, $i=1,2$, where $L(f)=\inf\{f(\vv x):\vv x  \in A^p\}$ and $U(f)=\sup \{f(\vv x):\ \vv x \in A^p\}$.
\end{Lem}
\begin{proof}
Suppose the Polish topology of $A$ is induced by  a metric $d$ and let $N(  x,\delta)=\{ y\in A:~d(  x,   y)<\delta\}$, $\delta>0$.  For any $\vv x=(x_1,\ldots,x_p)\in A^p$  and $\delta>0$, 
define  the product neighborhood (corresponding to the uniform metric on $A^p$ induced from $d$)
\[
N_p(\vv x,\delta)=N(x_1,\delta)\times 
\cdots
\times N(x_p,\delta).
\]
Let $C\subset A^p$ be the set of continuity points of $f$, and fix $\epsilon>0$.
For every $\vv x\in C$, when $\delta>0$ is small enough and avoids a countable set of values,  the set $N_p(\vv x,\delta)$ can be made elementary (i.e., each $N(x_i,\delta)$ is $\mu$-continuous, $i=1,\ldots,p$) and 
\[\omega(\vv x,\delta):=\sup\{ |f(\vv x)-f(\vv y)|:~ \vv y\in N_p(\vv x,\delta)\}<\epsilon.
\]   
Next, note that the  separable metric space $A^p$   is second-countable   and thus Lindel\"of (every open cover  has a countable subcover). Hence there exist $\delta_n>0$ and $\vv x_n\in C$, such that   $\cup_{n=1}^\infty N_p(\vv x_n,\delta_n) \supset C$, where each $N_p(\vv x_n,\delta_n)$  is elementary and  $\omega(\vv x_n,\delta_n)<\epsilon$. For each $m\in\N$, set 
$C_m:=\cup_{n=1}^m N_p(\vv x_n,\delta_n)$. This is an elementary set, and one can further choose $m$ large enough  so that $\mu^{\otimes p}(A^p\setminus C_m)=\mu^{\otimes p}(C\setminus C_m)<\epsilon$. One could further express $C_m$ as a union of disjoint elementary sets $C_m = \cup_{n=1}^mD_n$ with  $D_n:=N_p(\vv x_n,\delta_n)\setminus (\cup_{i=1}^{n-1}N_p(\vv x_i,\delta_i))$. 
Then define 
\[
g_1(\vv x) 
:=
  \sum_{n=1}^m
\inf\{f(\vv x):~\vv x\in D_n \} 1_{D_n}(\vv x) +\inf\{f(\vv x): ~ \vv x\in A^p\}1_{ A^p\setminus C_m }(\vv x) 
\]
and define   $g_2$ with   $\inf$'s replaced by $\sup$'s above. Then  $g_1$ and $g_2$ are elementary functions satisfying $g_1\le f \le g_2$, and
\[
  \mu^{\otimes p}(f-g_1) \wedge \mu_p (g_2-f)\ge 0, \quad   \mu^{\otimes p}(f-g_1) \vee \mu^{\otimes p}(g_2-f)\le \epsilon (\mu^{\otimes p}(A^p) + 2\|f\|_\infty).
\]
\end{proof}
\begin{proof}[Proof of Proposition \ref{Pro:unif ret}]
The ``only if'' part is immediate if  $\cl{C}$ to consists of all $\mu$-continuity sets in $\cl{E}_A$. We only need to show the ``if'' part.

  Let $\cl{D}$ be the smallest class of subsets of $A$ containing $\cl{C}$, which is also closed under (i) finite unions of disjoint sets and (ii)   proper set differences. Then we apply a variant of Dynkin's $\pi$-$\lambda$ theorem, where the $\sigma$-field is replaced by  a field, and in the definition of a $\lambda$-system, the   ``countable disjoint union'' is replaced by ``finite disjoint union''.     This variant  can be established using similar  arguments as those in \cite[Section 2.2.2]{resnick99probability}.   Applying this we conclude that $\cl{D}$ is the smallest field containing $\cl{C}$. On the other hand, the class of $\mu$-continuity subsets of  $A$  also forms a field, and so does $\cl{E}_A$.  Hence  any  set in $\cl{D}$ is $\mu$-continuous  and $\cl{D}\subset \cl{E}_A$.  
   Next, one can verify directly that the set operations (i) and (ii) mentioned above preserve \eqref{eq:unif ret C}, and hence the relation   \eqref{eq:unif ret C} holds for $B\in \cl{D}$.   

  Now note that $\mu$ restricted to Polish $A$ is tight   (see e.g. \cite[Theorem 1.3]{billingsley99convergence}). Hence for any $\mu$-continuity set $B\in \cl{E}_A$ and any $\epsilon>0$, there exists a compact $K\subset \mathring{B}$, such  that $\mu(B\setminus K)=\mu(\mathring{B}\setminus K)<\epsilon/2$. Due to the compactness and condition (b) of Proposition \ref{Pro:unif ret}, there exists $D_1\in \cl{D}$ which is a finite union of sets in $\cl{C}$, so that $ K \subset D_1\subset \mathring{B}$. This together with a similar argument with $B$ replaced by $A \setminus B$ entails the existence of $D_1,D_2\in \cl{D}$ satisfying $D_1\subset B \subset D_2$ and $\mu(D_2)-\mu(D_1)<\epsilon$.  Taking $n\rightarrow\infty$ in\begin{equation}\label{eq:approx two sides}
  b_n\wh{T}^n 1_{D_1} \le  b_n  \wh{T}^n 1_B\le b_n\wh{T}^n 1_{D_2} \quad \text{a.e.}, 
\end{equation}
we see that \eqref{eq:uniform ret} holds for  $g=1_B$.
To obtain \eqref{eq:uniform ret}  in   full generality,  first     observe that by linearity of $\wh{T}$, the relation  extends to $g$  which is a finite linear combination  of indicators of $\mu$-continuity sets in $\cl{E}_A$. Then it extends to general  bounded $\mu$-a.e.\  continuous $g$   by an approximation similar to \eqref{eq:approx two sides} via Lemma \ref{Lem:riemann approx} with $p=1$. 
\end{proof}

\subsection{Proof of convergence of finite-dimensional distributions}\label{sec:pf fdd}
We proceed by first writing 
\equh\label{eq:SR}
\ccbb{S_n(t)}_{t\in[0,1]} \eqd \ccbb{S_{n,m}(t)+R_{n,m}(t)}_{t\in[0,1]},
\eque
for $m\in\N$ with
\[
S_{n,m}(t)  :=\frac{1}{c_n}\summ k1{\floor {nt}} p!\sum_{ I \in \cl{D}_p(m)} \left(\prod_{i\in I}\varepsilon_i  \rho^{\leftarrow}(\Gamma_i/w_n)\right)  f\circ T_p^k(\vv U^{(n)}_I),
\]
where $\calD_p(m)$ is as in \eqref{eq:D_p(m)}. 
To show the convergence of  finite-dimensional distributions, we shall show \begin{equation}
S_{n,m}(t)\ConvFDD  \Gamma(\beta_p) \cdot p!\cdot\mu^{\otimes p}(f) \sum_{ I \in \cl{D}_p(m)} \left(\prod_{i\in I}\varepsilon_i  \Gamma_i^{-1/\alpha} \right) L_t\pp{\bigcap_{i\in I} (R_i+V_i)},\label{eq:Snm1}
\end{equation}
for all $m\in\N$ (compare it with \eqref{eq:Z_t [0,1]})
and
\equh\label{eq:Rn}
\limm \limsupn P(|R_{n,m}(t)|>\epsilon)=0, \mfa t\in[0,1], \epsilon>0.
\eque
We prove the two claims separately.
\subsubsection*{Proof of \eqref{eq:Snm1}.}
Introduce 
\[
G_n(y):=\frac{\rho^{\leftarrow}(y/w_n)}{\rho^{\leftarrow}(1/w_n)}
\]
and
\begin{equation}\label{eq:L_t(I,n)}
L_{n,I,t}:=\frac{ b_n^p}{n} \sum_{k=1}^{\lf nt \rf} f\circ T_p^k(\vv U_I^{(n)}),\quad I\in \cl{D}_p, t\ge0, n\in\N,
\end{equation}
and write
\equh
S_{n,m}(t)=p! \sum_{ I \in \cl{D}_p(m)} \left(\prod_{i\in I}\varepsilon_i  G_n(\Gamma_i)
\right)L_{n,I,t}
\label{eq:Snm0}.
\eque
By the assumption $\rho((x,\infty))\in\RV_\infty(-\alpha)$ we have that 
\[
\limn G_n(y) =  y^{-1/\alpha},\quad y>0.
\]
Therefore, \eqref{eq:Snm1} follows from the following result.
\begin{Thm}\label{Thm:local time}With the notation above,
\[
\pp{L_{n,I,t}}_{I\in \calD_p, t\in[0,1]} \stackrel{f.d.d.}\to
\mu^{\otimes p} (f)\Gamma(\beta_p)\pp{L_t\pp{\bigcap_{i\in I} (R_i+V_i)}}_{I\in \calD_p, t\in[0,1]}.
\]
\end{Thm}
Theorem \ref{Thm:local time} can be proved by a method of moments.
\begin{Pro}\label{Pro:moments}
Let $f$ be as in Theorem \ref{Thm:CLT}. Then
for any $I_1,\ldots,I_r \in \cl{D}_p$, $t_1,\ldots,t_r\in [0,1]$, we have 
\begin{align}\label{eq:moment sum}
\limn\E\pp{ \prod_{\ell =1}^r L_{n,I_\ell,t_{\ell}}}  = \mu^{\otimes p}(f)^{r}\int_{(\vv0,\vv t)} \prod_{i=1}^K \hIibeta (\xIi) \, d\vv x, 
\end{align}
where $h_{q}^{(\beta)}$ is as in \eqref{eq:hq} 
and $K = \max(\bigcup_{\ell=1}^r I_\ell)$. 
\end{Pro}

\begin{proof}
We may assume that $t_\ell >0$ for all $\ell=1,\ldots,r$, otherwise (\ref{eq:moment sum}) trivially holds with both-hand sides being zeros.
\comment{Let
\begin{equation}\label{eq:a_n^r}
a_n^{(r)}= \frac{n}{\bigl( \Gammabeta w_n \bigr)^r} \sim \frac{n}{b_n^r} ,\quad r=1,\ldots,p,
\end{equation}
where the asymptotic equivalence is due to  (\ref{eq:b_n w_n}).}
We then proceed as follows: 
\begin{align}
\esp\pp{\prodd \ell1r L_{n,I_\ell,t_\ell}} & = \pp{\frac{b_n^p}n}^r\esp\pp{\prodd\ell1r\summ k1{\floor{nt_\ell}}f\circ T_p^k(\vv U_{I_\ell}\topp n)}\notag\\
&
= \pp{\frac{b_n^p}n}^r \sum_{\vv 1\le \vv k\le \floor{n\vv t}} \E \pp{ \prod_{\ell=1}^r f\circ T_p^{k_\ell}(\vv U_{I_\ell}^{(n)}) }.  \label{eq:moment start}
\end{align}
We claim that it is enough to prove \eqref{eq:moment sum} for function $f$ of the form
\begin{equation}\label{eq:f single term}
f(\vv x) = \prod_{j=1}^p f_{j}(x_j),
\qmwith f_j(x) = \ind_{A_j}(x),
\end{equation}
where each $f_{j}$ is 
an indicator of a $\mu$-continuity set $A_j\in\calE_A$ 
 satisfying the uniform return relation  \eqref{eq:uniform ret} and \eqref{eq:bn_RV}. 
   Indeed, since  $f$ can always be written  as a difference of two  non-negative bounded $\mu^{\otimes p}$-a.e.\ continuous functions (e.g., 
 $f= (f + \| f \|_\infty 1_{A^p}) - \|f  \|_\infty 1_{A^p}$), so by an expansion of the product in \eqref{eq:moment start}, one may assume that $f\ge 0$. Next, in view of  Lemma \ref{Lem:riemann approx}, Assumption \ref{assump} and an approximation argument exploiting monotonicity, it suffices to consider $f$ which is elementary in the sense of Definition \ref{Def:eleme ntary}. By a further expansion of the product in \eqref{eq:moment start}, it suffices to focus on 
 $f$ with simple form
  \eqref{eq:f single term}.

From (\ref{eq:f single term}), we can rewrite   using $I_\ell = (I_\ell(1),\dots,I_\ell(p))$ with $I_\ell(1)<\cdots<I_\ell(p)$: 
\begin{align*} 
\prod_{\ell=1}^r f\circ T_p^{k_\ell}(U_{I_\ell}^{(n)})&= \prod_{\ell=1}^r \prod_{j=1}^p f_{j}\circ T^{k_\ell}(U_{I_\ell(j)}^{(n)})\notag\\
&=   \prod_{i=1}^K \prod_{\ell \in \Ii} f_{\Kil} \circ T^{k_\ell} (U_i^{(n)}),
\end{align*}
where, for every $\ell\in\calI(i) = \{\ell' \in\{1,\dots,r\}, i\in I_{\ell'}\}$, 
 $\Kil \in \{ 1,\dots, p \}$   is defined by the relation $I_\ell\big(\Kil\big)=i$. Here and below, we follow the convention $\prod_{\ell \in \emptyset}(\cdot) \equiv 1$.    Since $U_1^{(n)}, \dots, U_K^{(n)}$ are i.i.d.\ following $\mu_n$ in (\ref{eq:mu_n}), we have
$$
\E\pp{   \prod_{i=1}^K \prod_{\ell \in \Ii} f_{\Kil} \circ T^{k_\ell} (U_i^{(n)}) } = \prod_{i=1}^K \mu_n \pp{\prod_{\ell \in \Ii} f_{\Kil} \circ T^{k_\ell}  }. 
$$ 
Then,
\begin{align}
\esp\pp{\prodd \ell1r L_{n,I_\ell,t_\ell}} 
&=  \pp{\frac{b_n^p}n}^r\sum_{\vv1\le\vv k\le \floor{n\vv t}}\prodd i1K\mu_n\pp{\prod_{\ell\in\calI(i)}f_{\calK(i,\ell)}\circ T^{k_\ell}}.\label{eq:Psy_n}
\end{align}
Expressing the $r$-tuple sum over $\vv k$ above by an integral, we claim that
\begin{align}
&\esp\pp{\prodd \ell1r L_{n,I_\ell,t_\ell}}   = b_n^{pr}\int_{(\mathbf{0}, \lfloor n\vv t\rfloor /n)}\prodd i1K\mu_n\pp{\prod_{\ell\in\calI(i)}f_{\calK(i,\ell)}\circ T^{\floor{nx_\ell}+1}}d\vv x\nonumber\\
&\sim \pp{ \Gammabeta }^{pr}\int_{(\mathbf{0}, \lfloor n\vv t\rfloor /n)}\prodd i1Kw_n^{|\calI(i)|-1}\mu\pp{\prod_{\ell\in\calI(i)}f_{\calK(i,\ell)}\circ T^{\floor{nx_\ell}}}d\vv x.
\label{eq:DCT}
\end{align}
Indeed,  in \eqref{eq:DCT},  
we have used   $\mu_n(\cdot)=\mu(\cdot \cap \{\varphi\le n\})/w_n$, 
the relation \eqref{eq:b_n w_n}, and the fact that the functions $f_j\circ T^k$,
 $1\le k\le n$, are supported within $ \{\varphi\le n\}$ and   $\sum_{i=1}^K \abIi=|I_1|+\cdots+ |I_r|= pr$; we also drop the `$+1$' in the power of $T$, since $T$ is measure-preserving with respect to $\mu$. 

To complete the proof, it remains to establish
\begin{multline}\label{eq:DCT1}
\limn\int_{(\mathbf{0}, \lfloor n\vv t\rfloor /n)}\prodd i1Kw_n^{|\calI(i)|-1}\mu\pp{\prod_{\ell\in\calI(i)}f_{\calK(i,\ell)}\circ T^{\floor{nx_\ell}}}d\vv x \\
= \bigl( \Gammabeta \bigr)^{-pr}\left( \prod_{i=1}^K \prod_{\ell \in \Ii} \mu(f_{\Kil}) \right) \int_{(\vv0,\vv t)} \prod_{i=1}^K \hIibeta (\xIi) \, d\vv x.
\end{multline}
Indeed, 
the desired convergence of moments \eqref{eq:moment sum} now follows from 
\eqref{eq:Psy_n}, \eqref{eq:DCT}, \eqref{eq:DCT1} and that
$$
  \prod_{i=1}^K \prod_{\ell \in \Ii} \mu\left(f_{\Kil}\right) = \left( \prod_{j=1}^p \mu(f_{j})  \right)^r = \mu^{\otimes p}(f)^{r}.
$$

In order to show \eqref{eq:DCT1}, we apply the dominated convergence theorem. To simplify the notation, we consider $q \in\{ 1,\dots,p\}$ and $f_1,\dots,f_q$  as in \eqref{eq:f single term}, and introduce
\[
H_{n,q}(\vv x):= w_n^{q-1} \mu\left(  \prod_{j=1}^q f_j\circ T^{\lf n x_j \rf}\right), \quad \vv x\in (0,1)^q.
\]
A careful examination shows that \eqref{eq:DCT1} follows from the following two results:
\equh\label{eq:pointwise}
\limn H_{n,q}(\vv x)=    \pp{ \Gammabeta }^{-q} \left(\prod_{j=1}^q \mu(f_j) \right) \hqbeta (\vv x), \mfa \vv x\in(\vv0,\vv1)_{\neq},
\eque
and, for some $\eta\in(0,\beta)$, 
\equh\label{eq:upper_bound}
H_{n,q}(\vv x)\le C      h_q^{(\beta-\eta)} (\vv x), \mfa \vv x\in(\vv0,\vv1)_{\neq}.
\eque
(Recall $h_q\topp\beta$ in \eqref{eq:hq}.)
Note that we only need to consider the limit for $\vv x\in(\vv 0,\vv 1)_{\neq}:=\{\vv y\in(\vv0,\vv1):y_\ell\ne y_{\ell'}, \forall \ell\ne\ell'\}$. The product  $\prod_{i=1}^K h_{|\cl{I}(i)|}^{(\beta-\eta)} (\xIi)$ is integrable on $(\vv 0,\vv 1)_{\neq}$  since it is up to a multiplicative constant  
\[
\esp\pp{\prodd \ell1r\wt L_{I_\ell,t_\ell}} \le \frac1r \summ \ell1r \esp \wt L_{I_\ell,t_\ell}^r,
\]
where $\wt L_{I,t}$ is defined similarly as $L_{I,t}$, with the underlying $\beta$-stable regenerative sets replaced by $(\beta-\eta)$-stable regenerative sets (see \eqref{eq:LIt}). Setting $\eta>0$ small enough so that $p(\beta-\eta)-p+1\in(0,1)$, the finiteness of the integration now follows from  \eqref{eq:4.5}.

We now prove \eqref{eq:pointwise} and \eqref{eq:upper_bound}.  Assume        $q\ge 2$ below. The case $q=1$ is similar and simpler and hence omitted.
To show \eqref{eq:pointwise}, 
it suffices to focus on  the tetrahedron $(\vv 0,\vv1)_\uparrow:=\{\vv x\in(0,1)^q: 0 < x_1 < \dots < x_q < 1\}$. First write
\begin{align*}
\prodd j1q f_j\circ T^{\floor{nx_j}} & = f_{1}\circ T^{\floor{nx_1}} \times \pp{    \prod_{j=2}^q f_{j} \circ T^{\lf nx_{j} \rf - \lf nx_{1} \rf}} \circ T^{\lf nx_{1} \rf}\\
& = f_{1}\circ T^{\floor{nx_1}} \times \pp{    \prod_{j=2}^q f_{j} \circ T^{\lf nx_{j} \rf - \lf nx_{2} \rf}}\circ T^{\floor{nx_2}-\floor{nx_1}}\circ T^{\lf nx_{1} \rf}.
\end{align*}
Then, by the measure-preserving property,
\equh\label{eq:H_q,n}
 H_{n,q}(\vv x)  = w_n^{q-1} \int_E f_{1}\times \pp{    \prod_{j=2}^q f_{j} \circ T^{\lf nx_{j} \rf - \lf nx_{2} \rf}}\circ T^{\lf nx_{2} \rf - \lf nx_{1} \rf} d\mu,
\eque
which, by duality \eqref{eq:dual}, equals
\[
w_n^{q-2} \frac{w_n}{w_{\lf nx_{2} \rf - \lf nx_{1} \rf}}\, \int_A w_{\lf nx_{2} \rf - \lf nx_{1} \rf} \left(\That^{\lf nx_{2} \rf - \lf nx_{1} \rf} f_{1}\right) \prod_{j=2}^q f_{j}\circ T^{\lf nx_{j} \rf - \lf nx_{2} \rf} d\mu.
\]
Due to the uniform convergence of a regularly varying sequence of positive index  \citep[Proposition 2.4]{resnick07heavy}, we have 
$
\limn w_{\lf nx_{2} \rf - \lf nx_{1} \rf}/w_n = (x_{2} - x_{1})^{1-\beta}$.
In addition, using the uniform convergence in (\ref{eq:uniform ret}) and the relation (\ref{eq:b_n w_n}), 
as $n\to\infty$, 
\[
 H_{n,q}(\vv x) 
 \sim\frac{ \mu(f_{1})}{\Gammabeta } (x_{2} - x_{1})^{\beta-1} w_n^{q-2} \int_E  \prod_{j=2}^q f_{j} \circ T^{ \lf nx_{j} \rf - \lf nx_{2} \rf}d\mu. 
\]
Repeating the arguments above  yields \eqref{eq:pointwise}.

We now prove \eqref{eq:upper_bound}.  
The situation is more delicate, and we shall introduce 
\[
  D_{n,q}:=\ccbb{\vv x \in (\vv0,\vv1)_{\uparrow}:\lf n x_i \rf \neq \lf nx_j \rf \text{ for all }i\neq j }.
\]

First assume  that $\vv x\in D_{n,q}$, which implies $\lf n x_{1} \rf < \lf n x_{2} \rf$.   By the Potter's bound \citep[Theorem 1.5.6]{bingham87regular} and an elementary bound \citep[Eq.(40)]{bai14generalized},
\begin{equation}  \label{e:bound1 lem}
\frac{w_n}{w_{\lf nx_{2} \rf - \lf nx_{1} \rf}}   \leq  C_1 \left(\frac{\lf n x_{2} \rf - \lf n x_{1} \rf}{n} \right)^{\beta-1-\eta}\le C_2 (x_{2} - x_{1})^{\beta-1-\eta},  
\end{equation}
for all $n\in\N,\vv x\in D_{n,q}$,
where recall that $\eta>0$ is sufficiently small such that $\beta-\eta > 1-1/p$. 
In addition, the  relations (\ref{eq:uniform ret}) and (\ref{eq:b_n w_n}) imply
\begin{equation}  \label{e:bound2 lem}
\sup_{\substack{0 < x_1<x_2 < 1, y\in A\\ n: \lf n x_{1} \rf<\lf n x_{2}\rf}} w_{\lf nx_{2} \rf - \lf nx_{1} \rf} \left(\That^{\lf nx_{2} \rf - \lf nx_{1} \rf} 1_A\right) (y) <\infty. 
\end{equation}
Applying these observations to  (\ref{eq:H_q,n}), and bounding  $|f_j|$'s   by $1_A$ up to a constant almost everywhere, we get
\[
 H_{n,q}(\vv x)   \leq  C (x_{2}-x_{1})^{\beta-1-\eta} w_n^{q-2}
\ \int_E  1_A 
\prodd j3q
1_A\circ T^{\lf nx_{j} \rf - \lf nx_{2} \rf} d\mu.
\]
Applying the bounds of the form \eqref{e:bound1 lem} and \eqref{e:bound2 lem} iteratively, we eventually get 
\eqref{eq:upper_bound} for $\vv x\in D_{n,q}$.

Now we   assume  that $\vv x\in (\vv0,\vv1)_\uparrow\setminus D_{n,q}$.
Again in (\ref{eq:H_q,n}), we shall bound each  $|f_{j}|$ by $1_A$ up to a constant almost everywhere.  Assume first that only two of $\floor{nx_i}$s
 are the same, and without loss of generality we consider    $\lf nx_{2} \rf = \lf  n x_{1} \rf$ and $\lf n x_{j}\rf \neq \lf n x_{j-1} \rf $ for $j=3,\dots,q$.    Then
\begin{align*}
 H_{n,q}(\vv x) &\le C   w_n^{q-1} \int_E  \prod_{j=2}^q  1_A \circ T^{\lf nx_{j} \rf} d\mu \\
  &
  = C w_n  \cdot w_n^{q-2} \int_E 1_{A} \prod_{j=3}^q 1_{A} \circ T^{\lf nx_{j} \rf - \lf nx_{2} \rf} d\mu. 
\end{align*}
Handling the integral factor as  in \eqref{e:bound1 lem} and \eqref{e:bound2 lem}, we obtain 
\begin{equation}\label{eq:H_nq bound}
H_{n,q}(\vv x)\leq C w_n \prod_{j=3}^q (x_{j} - x_{{j-1}})^{\beta-1-\eta}\end{equation} 
Furthermore, since $\lf nx_{2} \rf = \lf  n x_{1}\rf $ implies $x_{2} - x_{1} < 1/n$, under which 
$ n^{\beta-1-\eta} (x_{2} - x_{1})^{\beta-1-\eta}>1$. Inserting this into \eqref{eq:H_nq bound}, 
it then follows that 
\[
H_{n,q}(\vv x) \leq C w_n n^{\beta-1-\eta} \hqbetaeta (\vv x).
\] 
Note that $w_n n^{\beta-1-\eta}\in \RV_\infty(-\eta)$ and thus converges to zero as $n\rightarrow\infty$. So the above satisfies  what we need in \eqref{eq:upper_bound}. The case where $\vv x\in (\vv0,\vv1)_{\uparrow}\setminus D_{n,q}$ with  $\floor{nx_i} = \floor{nx_{i+1}}$ more than one value of $i=1,\dots,q-1$ can be treated similarly. The proof is thus completed.
\end{proof}

\begin{proof}[Proof of Theorem \ref{Thm:local time}]
  We have computed the  joint moments of $(L_{I_\ell,t_\ell})_{\ell=1,\dots,r}$   in Theorem \ref{thm:1}. On the other hand, we have established the convergence of the joint   moments of $(L_{n,I_\ell,t_\ell})_{\ell=1,\dots,r}$ in Proposition \ref{Pro:moments}. It remains to show that the law of $(L_{I_\ell,t_\ell})_{\ell=1,\dots,r}$ is uniquely determined by the joint moments, for every choice of $I_1,\dots,I_r, t_1,\dots,t_r$.   
Then, it suffices to check the multivariate Carleman condition \citep[Theorem 1.12]{shohat43problem}
\begin{equation}\label{eq:carleman}
\sum_{k=1}^\infty  \eta_{2k}^{-1/(2k)}=\infty,
\qmwith 
\eta_{2k} := \summ \ell1r \esp L_{I_\ell,t_\ell}^{2k}.
\end{equation}
In view of Corollary \ref{Cor:incre moment}, we have
$\eta_{2k} \le   C^{2k} (2k)!/\Gamma(2k \beta_p -\beta_p
+2)$.
By the Stirling's approximation,  one  can obtain the inequality 
$\eta_{2k}^{-1/(2k)}\ge  C k^{\beta_p-1}$. 
So (\ref{eq:carleman}) holds  because $\beta_p>0$. 
\end{proof}
\subsubsection*{Proof of \eqref{eq:Rn}}
We shall need the following uniform control:
\begin{equation}\label{eq:Gn bound}
G_n(y)\equiv\frac{\rho^{\leftarrow}(y/w_n)}{\rho^{\leftarrow}(1/w_n)}\le  C \pp{ y^{-1/\alpha_0} +  y^{-(1/\alpha)-\epsilon}},~  \mfa y>0  \text{ and } n\in\N.
\end{equation}
To see this, we first note that the assumptions on $\rho$ in \eqref{eq:rho} imply that $\rho^\leftarrow\in\RV_0(-1/\alpha)$ and $\rho^\leftarrow(y) = O(y^{-1/\alpha_0})$ as $y\to\infty$. By Potter's bound \citep[Theorem 1.5.6]{bingham87regular}, for every $\epsilon>0$ there exists a constant $A_\epsilon>0$ such that If $y\le A_\epsilon w_n$, $G_n(y)\le 2y^{-(1/\alpha) - \epsilon}$. 
On the other hand, for $y>A_\epsilon w_n$, we have $\rho^\leftarrow(y/w_n)\le C(y/w_n)^{-1/\alpha_0}$  and $\rho^\leftarrow(1/w_n)\ge C(1/w_n)^{-(1/\alpha)+\epsilon}$, whence  we have
\[
G_n(y)\le C y^{-1/\alpha_0}w_n^{1/\alpha_0-(1/\alpha)+\epsilon}, \mfa y>A_\epsilon w_n, n\in\N.
\]
(The constants $C$ here and below depend on $\epsilon$.)
Now, note that for the second assumption on $\alpha_0$ in \eqref{eq:rho}, one could take $\alpha_0$ arbitrarily close to and smaller than 2. Set also $\epsilon$ small so that $1/\alpha_0-(1/\alpha)+\epsilon<0$, so that the upper bound above 
becomes
 $G_n(y)\le Cy^{-1/\alpha_0}$ for all $y>A_\epsilon w_n$. We have thus proved \eqref{eq:Gn bound}. 

Fix a large $M$ which will be specified later. In view of \eqref{eq:SR} and \eqref{eq:Snm0}, we 
express
\[
R_{n,m}(t)=\sum_{I_1\in  \cl{D}_{\le  p-1}(M)}   \left( \prod_{i\in I_1} \varepsilon_i G_{n}(\Gamma_i)\right) F(I_1,n,M,m)
\]
 where $\cl{D}_{\le p-1}(M)$ is as in (\ref{eq:D le p M}), and  
\[
F(I_1,n,M,m) := \sum_{I_2\in \cl{H}(p-|I_1|,M,m)}\left(\prod_{i\in I_2} \varepsilon_i G_{n}(\Gamma_i)\right)  L_{n,I_1\cup I_2,t},
\]
with
\[
\cl{H}(k,M,m):=\{I\in \cl{D}_k: \ \min I> M,\ \max I>m\}.
\]
(Compare it with $\calH(k,M)$ in \eqref{eq:H(k,M)}.)
Observe that $\cl{D}_{\le p-1}(M)$ is finite and  $\E |\prod_{i\in I_1}G_{n}(\Gamma_i)|^q<\infty$ for all $I_1 \in \cl{D}_{\le p-1}(M)$ when $q>0$ is sufficiently small  in view of \eqref{eq:Gn bound}   and \eqref{eq:gamma estimate}. Hence  by 
 H\"older's inequality, it 
 suffices
  to show for each $I_1\in \cl{D}_{\le p-1}(M)$,    
\begin{equation}\label{eq:I_1 term q norm b}
\limm \sup_{n\in\N} \E F(I_1,n,M,m)^2=0.
\end{equation}
For the above to hold we shall actually need $M$ to be large enough, which will be determined at the end. 
Introduce
\[
k:= p-|I_1|.
\] 
We start by using the orthogonality  $\E[(\prod_{i\in I}\varepsilon_i)( \prod_{i\in I'} \varepsilon_i)] =1_{\{I=I'\}}$, $I,I'\in 
\calD_k$
 to obtain
\[
\E F(I_1,n,M,m)^2=\sum_{I_2\in \cl{H}(k,M,m)} \E\left( \prod_{i\in I_2}  G_{n}(\Gamma_i)^2\right)  \E L_{n,I_1\cup I_2,t}^2.
\]
Note that $\esp L_{n,I_1\cup I_2,t}^2 = \esp L_{n,I,t}^2$ for all $I\in \calD_p$, which is convergent as $n\to\infty$ by Proposition \ref{Pro:moments}  and hence uniformly bounded in $I$ and $n$. Note also that $\cl{H}(k,M,m)\downarrow \emptyset$ as $m\rightarrow\infty$.
Therefore, to show \eqref{eq:I_1 term q norm b}, by the dominated convergence theorem it suffices to find $g^*: \calH(k,M)\to\R_+$ such that
\[
g_n^*(I_2):=\esp\pp{\prod_{i\in I_2}G_n(\Gamma_i)^2}\le g^*(I_2), \mfa I_2\in \calH(k,M), n\in\N
\]
and $\sum_{I_2\in\calH(
k
,M)} g^*(I_2)<\infty$.
Setting $\gamma:=\min\{1/\alpha_0,
1/\alpha
+\epsilon\}$ 
and taking $M>2\gamma k$,
 we have
\equh
\esp\pp{\prod_{i\in I_2}G_n(\Gamma_i)^2}
  \le C\esp \pp{\prod_{i\in I_2}\pp{\Gamma_i^{-1/\alpha_0}+\Gamma_i^{-(1/\alpha)-\epsilon}}^2}
\le C{\prod_{i\in I_2} i^{-2\gamma}}=:g^*(I_2),
\label{eq:Gn_Gamma}
\eque
where the first inequality follows from \eqref{eq:Gn bound}, and the second from  \eqref{eq:gamma estimate}. 
 The bound $g^*$ is summable over $\calH(k,M)$ as
\[
\sum_{I_2\in\calH(k,M)}g^*(I_2)\le C\pp{\sum_{i=1}^\infty i^{-2\gamma}}^{k},
\]
and that $2\gamma>1$.
This completes the proof of \eqref{eq:I_1 term q norm b} and hence \eqref{eq:Rn}.
\subsection{Proof of tightness}\label{sec:pf tight}
\begin{Pro}
Under the assumptions of Theorem \ref{Thm:CLT}, the laws of processes $(S_n(t))_{t\in[0,1]}, n\in\N$ are tight in the Skorokhod space $D([0,1])$ with respect to the uniform topology. 
\end{Pro}

\begin{proof}
Fix   $m\in \N$  large enough 
specified later. 
Assume without   loss of generality that $f\ge 0$, since a general $f$ can be written as a difference of two non-negative bounded $\mu^{\otimes p}$-a.e.~continuous functions on $A^p$. 
 Recall the decomposition $S_n(t)=S_{n,m}(t)+R_{n,m}(t)$ as in (\ref{eq:SR}).
It suffices to check the tightness of $(S_{n,m})_{n\in\N}$ and $(R_{n,m})_{n\in\N}$ respectively. We start with $(S_{n,m})_{n\in\N}$. 
Let $L_{n,I,t}$ be as in  (\ref{eq:L_t(I,n)}). Recall that 
\[
S_{n,m}(t) = p!\sum_{ I \in \cl{D}_p(m)}\pp{\prod_{i\in I}\varepsilon_iG_n(\Gamma_i)}L_{n,I,t}.
\]
By Theorem \ref{Thm:local time},  the limit of each $L_{n,I,t}$ in finite-dimensional distribution is, up to a constant, the local time $L_t(\cap_{i\in I}(R_i+V_i))$ of the shifted $\beta_p$-stable regenerative set $\cap_{i\in I}(R_i+V_i)$, 
 for which we shall work with its continuous version.  Then for each fixed $I\in \cl{D}_p(m)$,  the laws  of the a.s.\ non-decreasing processes  $(L_{n,I,t})_{t\in[0,1]}, n\in\N$ are tight \citep[Theorem 3]{bingham71limit}.     Furthermore, we have seen that $\prod_{i\in I}G_{n}(\Gamma_i)\to \prod_{i\in I}\Gamma_i^{-1/ \alpha}$ as $n\rightarrow\infty$, and hence \[
\wt G_{n,I}:=\prod_{i\in I}\varepsilon_iG_n(\Gamma_i), n\in\N
\] is a tight sequence of random variables for every $I\in\calD_p(m)$. For every fixed $m\in\N$, the tightness of $\{(S_{n,m}(t))_{t\in[0,1]}, n\in\N\}$ then follows.

Next, we show the tightness of $(R_{n,m}(t))_{t\in[0,1]}, n\in\N$ for $m$ fixed large enough. 
Write
\[
R_{n,m}(t)
=\sum_{I_1\in  \cl{D}_{\le p-1}(m)}  \wt G_{n,I_1}
{\sum_{I_2\in \cl{H}(p-|I_1|,m)}
\wt G_{n,I_2}
   L_{n,I_1\cup I_2,t}},
\]
Since $\cl{D}_{\le p-1}(m)$ is finite, it suffices to 
prove, for fixed $I_1\in \cl{D}_{\le p-1}(m)$ and $k=p-|I_1|\ge 1$, the tightness of   
\[
A_n(t):=\sum_{I_2\in \cl{H}(k,m)}\wt G_{n,I_2}  L_{n,I_1\cup I_2,t}, t\in[0,1], n\in\N.
\]
For this purpose, it is standard (e.g.~\citep[Theorem 13.5]{billingsley99convergence}) to show that for all $0\le s<t\le 1$, there exist  constants $C>0$, $a>0$ and $b>1$, such that
\begin{equation}\label{eq:goal tight}
\E |A_n(t)-A_n(s)|^{a} \le C \left( t-s \right)^{b}, \mfa 0\le s<t\le 1, n\in\N.
\end{equation}
For this purpose, we compute
\begin{align*}
  \E (A_n(t)-A_n(s))^{2}  
=   \sum_{I_2 \in \cl{H}(k,m)}  \E \pp{\prod_{i\in I_2}G_n(\Gamma_i)^{2}}  \E 
(L_{n,I_1\cup I_2,t}-L_{n,I_1\cup I_2,s})^2.
\end{align*} 
The first expectation is uniformly bounded by $g^*(I_2)$  as in \eqref{eq:Gn_Gamma} (assuming $m>2\gamma k$ in place of $M>2\gamma k$), which is summable over $\calH(k,m)$. For the second, by first bounding $f$ by $1_{A^p}$ up to a constant and then  applying an argument similar to the proof of Proposition \ref{Pro:moments},   in particular,  using the bound \eqref{eq:upper_bound}, we have   
\begin{align*}
\E  
(L_{n,I_1\cup I_2,t}-L_{n,I_1\cup I_2,s})^2 &\le C  \int_{ \frac{\lf ns \rf}{n}<x_1<x_2< \frac{\lf nt \rf}{n} } (x_2-x_1)^{p(\beta-1-\eta)} dx_1 dx_2   \\
&
\le C \left(s-t\right)^{\beta_p +1 -p\eta},
\end{align*}
where $\eta>0$ is chosen sufficiently small so that $p\eta<\beta_p\in (0,1)$.
The proof of (\ref{eq:goal tight}) is then completed.
\end{proof}

\section*{Acknowledgements}
We would like to thank Gennady Samorodnitsky for helpful discussions.

TO's research was partially supported by the National Science Foundation (NSF) grant, Probability and Topology \#1811428.
YW's research was partially supported by Army Research Laboratory grant W911NF-17-1-0006 and National Security Agency (NSA) grant H98230-16-1-0322.



\end{document}